\documentclass[11pt]{amsart}

\usepackage[paperheight=11in, 
    paperwidth=8.5in,
    outer=1in,
    inner=1in,  
    bottom=1in,
    top=1in]{geometry}

\usepackage[all]{xy}
\usepackage{amsmath, amscd, amsthm, amsfonts, amssymb, mathrsfs}
\usepackage{graphicx}
\usepackage{comment}
\usepackage{url}
\usepackage{color}
\usepackage{bbm}
\usepackage{tikz-cd}

\usepackage[T1]{fontenc}
\usepackage{enumitem}
\usepackage[hidelinks]{hyperref} 
 \hypersetup{
    colorlinks,
    citecolor=magenta,
    filecolor=magenta,
    linkcolor=blue,
    urlcolor=black
}
\usepackage{cleveref}

\theoremstyle{definition}
\newtheorem{thm}{Theorem}[section]
\newtheorem{cor}[thm]{Corollary}
\newtheorem{lem}[thm]{Lemma}
\newtheorem{prop}[thm]{Proposition}
\newtheorem{defn}[thm]{Definition}
\newtheorem{eg}[thm]{Example}
\newtheorem{rem}[thm]{Remark}
\newtheorem{conj}[thm]{Conjecture}
\newtheorem{ques}[thm]{Question}

\newcommand{\QQ}{\mathbb Q}
\newcommand{\RR}{\mathbb R}
\newcommand{\CC}{\mathbb C}
\newcommand{\PP}{\mathbb P}
\newcommand{\ZZ}{\mathbb Z}
\renewcommand{\AA}{\mathbb A}
\def\OO{\mathcal{O}}
\def\MM{{\mathbf M}}
\def\BB{{\mathbf B}}
\def\SS{{\mathbf S}}
\def\BBi{{\BB_\lambda}}

\def\fP{\mathcal{P}}

\def\bz{{\boldsymbol{\zeta}}}
\def\be{\mathbf{e}}
\def\br{\mathbf{r}}

\def\bt{\mathbf{t}}
\def\bk{\mathbf{k}}
\def\pp{\mathfrak p}
\def\qq{\mathfrak q}
\def\mm{\mathfrak m}

\DeclareMathOperator{\Cone}{Cone}

\DeclareMathOperator{\Hilb}{Hilb}
\DeclareMathOperator{\One}{\mathbbm{1}}
\DeclareMathOperator{\rk}{rk}
\newcommand{\KT}{K\mathcal T}

\title{K-theoretic Tutte polynomials of morphisms of matroids}
\author{Rodica Dinu, Christopher Eur, Tim Seynnaeve}

\address{Faculty of Mathematics and Computer Science, University of Bucharest, Str. Academiei 14, 010014, Bucharest, Romania}
\email{rdinu@fmi.unibuc.ro}
\address{Stanford University. Stanford, CA. USA.}
\email{chriseur@stanford.edu}
\address{Universit{\"a}t Bern, Mathematisches Institut, Alpeneggstrasse 22, 3012 Bern,Switzerland}
\email{tim.seynnaeve@math.unibe.ch}

\begin{document}

\maketitle

\begin{abstract}
We generalize the Tutte polynomial of a matroid to a morphism of matroids via the K-theory of flag varieties.  We introduce two different generalizations, and demonstrate that each has its own merits, where the trade-off is between the ease of combinatorics and geometry.  One generalization recovers the Las Vergnas Tutte polynomial of a morphism of matroids, which admits a corank-nullity formula and a deletion-contraction recursion.  The other generalization does not, but better reflects the geometry of flag varieties.
\end{abstract}

\section{Introduction}

Matroids are combinatorial abstractions of hyperplane arrangements that have been fruitful grounds for interactions between algebraic geometry and combinatorics.  One interaction concerns the Tutte polynomial of a matroid, an invariant first defined for graphs by Tutte \cite{Tut67}
and then for matroids by Crapo \cite{Cra69}.

\begin{defn}
Let $M$ be a matroid of rank $r$ on a finite set $[n] = \{1,2,\ldots, n\}$ with the rank function $\operatorname{rk}_M: 2^{[n]}\to \ZZ_{\geq 0}$.  Its \textbf{Tutte polynomial} $T_M(x,y)$ is a bivariate polynomial in $x,y$ defined by
\[
T_M(x,y) := \sum_{S \subseteq [n]} (x-1)^{r - \operatorname{rk}_M(S)}(y-1)^{|S| - \operatorname{rk}_M(S)}.
\]
\end{defn}

An algebro-geometric interpretation of the Tutte polynomial was given in \cite{FS12} via the $K$-theory of the Grassmannian.  Let $Gr(r;n)$ be the Grassmannian of $r$-dimensional linear subspaces in $\CC^n$, and more generally let $Fl(\br ; n)$ be the flag variety of flags of linear spaces of dimensions $\br = (r_1, \ldots, r_k)$.
The torus $T = (\CC^*)^n$ acts on $Gr(r;n)$ and $Fl(\br;n)$ by its standard action on $\CC^n$.  A point $L\in Gr(r;n)$ on the Grassmannian corresponds to a realization of a matroid, and its torus-orbit closure defines a $K$-class $[\OO_{\overline{T\cdot L}}]\in K^0(Gr(r;n))$ that depends only on the matroid.  In general, a matroid $M$ of rank $r$ on $\{1, \ldots, n\}$ defines a $K$-class $y(M) \in K^0(Gr(r;n))$.  Fink and Speyer related $y(M)$ to the Tutte polynomial $T_M(x,y)$ via the diagram
\begin{equation}\label{eqn:introFMGr}
\begin{tikzcd}
 &Fl(1,r,n-1;n) \arrow[ld, "\pi_r"'] \ar[rd, "\pi_{(n-1)1}"] & &\\
Gr(r;n) & & Gr(n-1;n) \times Gr(1;n) \arrow[r, equal] & (\PP^{n-1})^\vee\times  \PP^{n-1},
\end{tikzcd}
\end{equation}
where $\pi_r$ and $\pi_{(n-1)1}$ are maps that forget appropriate subspaces in the flag.

\begin{thm} \cite[Theorem 5.1]{FS12}
Let $\OO(1)$ be the line bundle on $Gr(r;n)$ of the Pl\"ucker embedding $Gr(r;n)\hookrightarrow \PP^{\binom{n}{r}-1}$. With notations as above, we have
\[
T_M(\alpha,\beta) = (\pi_{(n-1)1})_* \pi_r^* \Big( y (M) \cdot [\OO(1)]\Big) \in K^0((\PP^{n-1})^\vee\times \PP^{n-1})\simeq \QQ[\alpha,\beta]/(\alpha^n, \beta^n),
\]
where $\alpha, \beta$ are the $K$-classes of the structures sheaves of the hyperplanes of $(\PP^{n-1})^\vee, \PP^{n-1}$.
\end{thm}

We extend this relation between matroids and the $K$-theory of Grassmannians to a relation between flag matroids and the $K$-theory of flag varieties.  As a result, we show that there are (at least) \emph{two} different generalizations of the Tutte polynomial to flag matroids, each with its own merits.

\medskip
A \textbf{flag matroid} is a sequence of matroids $\MM = (M_1, \ldots, M_k)$ on a common ground set such that every circuit of $M_i$ is a union of circuits of $M_{i-1}$ for all $2\leq i \leq k$.  The \textbf{rank} of $\MM$ is the sequence $(\operatorname{rk}(M_1), \ldots, \operatorname{rk}(M_k))$.  Most of this paper will concern the case of $k=2$.  In this case, the two-step flag matroids $(M_1,M_2)$ are often called \textbf{matroid morphisms} or \textbf{matroid quotients}.  They are combinatorial abstractions of graph homomorphisms, linear surjections, and embeddings of graphs on surfaces.  See \S\ref{subsection:flagmat} for details on flag matroids and matroid quotients.

\medskip
Many features of matroids naturally generalize to flag matroids.  For instance, just as a point on a Grassmannian corresponds to a realization of a matroid, a point $\mathbf L$ on the flag variety $Fl(\br ;n)$ corresponds to a realization of a flag matroid.  The torus-orbit closure of $\mathbf L$ defines a $K$-class $[\OO_{\overline{T\cdot \mathbf L}}] \in K^0(Fl(\br;n))$ that depends only on the flag matroid.  In general, a flag matroid $\MM$ of rank $\br$ on a ground set $\{1, \ldots, n\}$ defines a $K$-class $y(\MM)$ of the flag variety $Fl(\br;n)$.  See \S\ref{subsection:flagmatKclass} or \cite[\S8.5]{CDMS20} for details.

\medskip
At this point, however, extending the constructions on matroids to flag matroids splits into several strands, for there are (at least) two distinguished ways to generalize the diagram \eqref{eqn:introFMGr}.
\begin{itemize}
\item The "flag-geometric" diagram:
\[\label{construction1:intro}\tag{$Fl$}
\begin{tikzcd}
 &Fl(1,\br,n-1;n) \arrow[ld, "\pi_\br"'] \arrow[rd, "\pi_{(n-1)1}"] &\\
Fl(\br;n) & & (\PP^{n-1})^\vee \times \PP^{n-1}
\end{tikzcd}
\]
where $\pi_\br$ and $\pi_{(n-1)1}$ are maps that forget appropriate subspaces in the flag.

\smallskip
\item The "Las Vergnas" diagram:
\[\label{construction2:intro}\tag{$\widetilde{Fl}$}
\begin{tikzcd}
 &\widetilde{Fl}(1,\br,n-1;n) \ar[ld, "\widetilde\pi_\br"'] \ar[rd, "\widetilde\pi_{(n-1)1}"] &\\
Fl(\br;n) & & (\PP^{n-1})^\vee \times \PP^{n-1}
\end{tikzcd}
\]
where $\widetilde{Fl}(1,\br,n-1;n)$ is the variety defined as
\[ 
\widetilde{Fl}(1,\br,n-1;n) := \left\{\left.\begin{matrix}\textnormal{linear subspaces}\\ (\ell, L_1, \ldots, L_k, H)\end{matrix} \ \right| \ \begin{matrix}\textnormal{$\dim \ell = 1$, $\dim H = n-1$, $(L_1, \ldots, L_k) \in Fl(\br;n)$,} \\ \textnormal{ and $\ell \subseteq L_k$ and $L_1 \subseteq H$} \end{matrix}\right\},
\]
and $\widetilde\pi_\br$ and $\widetilde\pi_{(n-1)1}$ are maps that forget appropriate subspaces in the flag.
\end{itemize}

Let us first consider the construction \eqref{construction2:intro}.  While the construction \eqref{construction2:intro} may seem geometrically unnatural, since $\widetilde{Fl}(1,\br,n-1;n)$ is not a flag variety, it leads to the previously established notion of Las Vergnas' Tutte polynomials of morphisms of matroids, defined as follows.

\begin{defn}
Let $\MM = (M_1, M_2)$ be a two-step flag matroid on a ground set $[n]= \{1, \ldots, n\}$.  For $i = 1,2$, let us write $r_i$ for the rank of $M_i$, and $r_i(S)$ for the rank of $S\subseteq [n]$ in $M_i$.  The \textbf{Las Vergnas Tutte polynomial} of $(M_1, M_2)$ is a polynomial in three variables $x,y,z$ defined by
\[
LV\mathcal T_\MM(x,y,z) := \sum_{S \subseteq [n]} (x-1)^{r_1 - r_1(S)}(y-1)^{|S| - r_2(S)} z^{r_2 - r_2(S) - (r_1 - r_1(S))}.
\]
\end{defn}

Las Vergnas introduced this generalization of the Tutte polynomial in \cite{LV75}, and studied its properties in a series of subsequent works \cite{LV80, LV84, LV99, ELV04, LV07, LV13}.  Our first main theorem is a $K$-theoretic interpretation of the Las Vergnas Tutte polynomial.

\newtheorem*{thm:LVT}{\Cref{thm:LVT}}
\begin{thm:LVT}
Let $\MM = (M_1, M_2)$ be a flag matroid with $r_1 = \operatorname{rk}(M_1)$, $r_2 = \operatorname{rk}(M_2)$.  Let $\OO(0,1)$ be the line bundle on $Fl(r_1,r_2;n)$ of the map $Fl(r_1,r_2;n) \to Gr(r_2;n) \hookrightarrow \PP^{\binom{n}{r_2}-1}$, and let $\mathcal S_2/\mathcal S_1$ be the vector bundle on $Fl(r_1,r_2;n)$ whose fiber over a point $(L_1,L_2)\in Fl(r_1,r_2;n)$ is $L_2/L_1$.  Then,
\[
LV\mathcal T_{\MM}(\alpha, \beta, w) = \sum_{m=0}^{r_2 -r_1} ({\widetilde \pi}_{(n-1)1})_* {\widetilde \pi}_\br^* \Big(  y(\MM) [\OO(0,1)][\bigwedge^m(\mathcal S_2/\mathcal S_1)] \Big)w^m
\]
as elements in $K^0((\PP^{n-1})^\vee\times \PP^{n-1})[w] \simeq \QQ[\alpha,\beta,w]/(\alpha^n, \beta^n)$.
\end{thm:LVT}

Let us now consider the construction \eqref{construction1:intro}.  It leads to the following different generalization of the Tutte polynomial, first defined in the review \cite{CDMS20}.

\newtheorem*{defn:KTdefn}{\Cref{defn:KTdefn}}
\begin{defn:KTdefn}
Let $\MM$ be a flag matroid of rank $\br = (r_1, \ldots, r_k)$ on $\{1, \ldots, n\}$, and let $\OO(\mathbf 1)$ be the line bundle of the embedding $Fl(\br;n) \hookrightarrow Gr(r_1;n) \times \cdots \times Gr(r_k;n) \hookrightarrow \PP^{\binom{n}{r_1}-1}\times \cdots\times \PP^{\binom{n}{r_k}-1}$. The \textbf{flag-geometric Tutte polynomial} of $\MM$, denoted $\KT_\MM(x,y)$, is the unique bivariate polynomial in $x,y$ of bi-degree at most $(n-1,n-1)$ such that
\[
\KT_\MM(\alpha, \beta)  =  (\pi_{(n-1)1})_* \pi_r^* \Big( y (\MM) \cdot [\OO(\mathbf 1)]\Big) \in K((\PP^{n-1})^\vee\times \PP^{n-1})\simeq \QQ[\alpha,\beta]/(\alpha^n, \beta^n).
\]
\end{defn:KTdefn}

While the construction \eqref{construction1:intro} may be more geometrically natural than the construction \eqref{construction2:intro}, we show that the flag-geometric Tutte polynomial $\KT_\MM$ fails to display the two characteristic combinatorial properties of the usual Tutte polynomial that the Las Vergnas Tutte polynomial satisfies.

\medskip
The first property is the "corank-nullity formula": 
Both the usual Tutte polynomial and the Las Vergnas Tutte polynomial can be expressed as a summation over all subsets of the ground set, with terms involving coranks and nullities of subsets.  As a result, for a matroid $M$ or a flag matroid $(M_1,M_2)$ on $\{1, \ldots, n\}$, one has $T_M(2,2) = 2^n$ and $LV\mathcal T_{(M_1, M_2)}(2,2,1) = 2^n$.  We show that the value of $\KT_\MM(2,2)$ is more intricate.

\newtheorem*{thm:KT22}{Theorem \ref{thm:KT22}}
\begin{thm:KT22}
Let $\MM$ be a two-step flag matroid $\MM = (M_1,M_2)$ on a ground set $[n] = \{1, \ldots, n\}$.  Let $p\mathcal B(\MM)$ be the set of subsets $S\subseteq [n]$ such that $S$ is spanning in $M_1$ and independent in $M_2$.  Then with $q$ as a formal variable, we have  
\[
 \KT_\MM(1+q^{-1},1+q) = q^{-r_2} \cdot (1+q)^n \cdot \Big(\sum_{S \in p\mathcal B(\MM)} q^{|S|}\Big),\quad \textnormal{in particular,}\quad \KT_\MM(2,2) =  2^n\cdot |p\mathcal B(\MM)|.
\]
\end{thm:KT22}

The second property is the deletion-contraction recursion that the Tutte polynomial and the Las Vergnas Tutte polynomial both satisfy.  Unlike them, the flag-geometric Tutte polynomial $\KT_\MM$ does not satisfy the usual deletion-contraction recursion.
We instead show the following deletion-contraction-like relation.

\newtheorem*{thm:delcont}{\Cref{thm:delcont}}
\begin{thm:delcont}
Let $M$ be a matroid on a ground set $\{0,1,\ldots, n\}$ such that the element $0$ is neither a loop nor a coloop in $M$.  Then we have
\[
\KT_{(M,M)}(x,y) =  \KT_{(M/0,M/0)}(x,y) + \KT_{(M/0,M\setminus 0)}(x,y)  + \KT_{(M\setminus 0, M\setminus 0)}(x,y).
\]
\end{thm:delcont}

\medskip
The three main theorems \Cref{thm:LVT}, \Cref{thm:KT22}, and \Cref{thm:delcont} are obtained by proving stronger torus-equivariant versions of the statements.  The resulting torus-equivariant statements are then reduced to computing certain summations of lattice point generating functions, techniques for which we review, extend, and specialize in Section~\S\ref{section:latticegenfct}.  Our main contribution here is \Cref{thm:sumOfCones}, which serves as a key technical tool in this paper and may be of independent interest in the study of lattice polyhedra.

The reduction to lattice point generating functions is done via the method of equivariant localization, reviewed in Section~\S\ref{subsection:GKZ}, aided by certain push-pull computations in Section~\S\ref{section:fundcomp}.  Section ~\S\ref{section:prelim} is largely a summary of a more detailed account \cite[\S8]{CDMS20} on flag matroids and the torus-equivariant $K$-theory of flag varieties.  We discuss some future directions in Section~\S\ref{section:future}.

\subsection{Computation}\label{subsection:computer}
At \url{https://github.com/chrisweur/kTutte}, the reader can find a Macaulay2 code for computations with torus-equivariant $K$-classes and flag matroids.  In particular, it computes the polynomials $LV\mathcal T_{\MM}$ and $\KT_\MM$ and their torus-equivariant versions.

\subsection{Notation} Throughout we set $[n] := \{1,\ldots, n\}$.  For $i=1,\ldots, n$, we set $\be_i$ to be the standard coordinate vector in $\RR^n$ (or $\CC^n$), and write $\be_S:=\sum_{i\in S} \be_i$ for a subset $S\subseteq [n]$.  Let $\langle \cdot, \cdot \rangle$ be the standard inner product on $\RR^n$.  Cardinality of a set $S$ is denoted by $|S|$, and disjoint unions by $\sqcup$.  A variety is a reduced and irreducible proper scheme over $\CC$.

\section{Preliminaries: flag matroids and their $K$-classes on flag varieties}\label{section:prelim}

Here we review flag matroids and their (torus-equivariant) $K$-classes on flag varieties.  Most of the material in this section is described in more detail in the review \cite{CDMS20}.

\subsection{Matroid quotients and flag matroids}\label{subsection:flagmat}

We assume familiarity with the fundamentals of matroid theory, and point to \cite{Oxl11,Whi86, Wel76} as references.  We write $U_{r,n}$ for the uniform matroid of rank $r$ on $[n]$.  For a linear subspace $L\subseteq \CC^n$, let $M(L)$ denote the linear matroid whose ground set is the image of $\{\be_1, \ldots, \be_n\}$ under the dual map $\CC^n \twoheadrightarrow L^\vee$.  For a matroid $M$ on a ground set $[n]$ we set:
\begin{itemize}
\item $\operatorname{rk}_M: 2^{[n]} \to \ZZ$ to be the rank function of $M$, with $r(M) := \operatorname{rk}_M([n])$,
\item $M|S$, $M \setminus S$ and $M/S$ to be the restriction to, the deletion of, and the contraction by a subset $S\subseteq [n]$ (respectively),
\item $\mathcal B(M)$ to be the set of bases of $M$, and
\item $Q(M) \subset \RR^n$ to be the base polytope of $M$, which is the convex hull of $\{\be_B \mid B\in \mathcal B(M)\}$.
\end{itemize}

In this paper, by \textbf{morphisms of matroids} we will mean matroid quotients, as defined below\footnote{The behavior of a morphism of matroids, in a more general sense of \cite{EH20} or \cite{HP18}, is largely governed by an associated matroid quotient \cite[Lemma 2.4]{EH20}.}.  They are combinatorial abstractions of the graph homomorphisms, linear maps, and graphs embedded on surfaces; see \cite{EH20} for illustrations of these examples.

\begin{defn} \label{defn:matroidquotient}
Let $M_1, M_2$ be two matroids on a common ground set $[n]$.  We say that $M_1$ is a \textbf{matroid quotient} of $M_2$, written $M_1 \twoheadleftarrow M_2$, if any of the following equivalent conditions are met \cite[Proposition 7.4.7]{Bry86}:
\begin{enumerate}
\item every circuit of $M_2$ is a union of circuits of $M_1$,
\item \label{item:matroidquotientrank} $\operatorname{rk}_{M_2}(B) - \operatorname{rk}_{M_2}(A) \geq \operatorname{rk}_{M_1}(B) - \operatorname{rk}_{M_1}(A)$ for any $A\subseteq B \subseteq [n]$,
\item there exists a matroid $N$ on a ground set $[n]\sqcup S$ with $|S| = r(M_2) - r(M_1)$ such that $M_1 = N/S$ and $M_2 = N\setminus S$.
\end{enumerate}
\end{defn}

\begin{eg} Matroid quotients are combinatorial abstractions of linear maps of maximal rank.  An inclusion of linear subspaces $L_1 \hookrightarrow L_2 \subseteq \CC^n$, or equivalently a quotient $\CC^n \twoheadrightarrow L_2^\vee \twoheadrightarrow L_1^\vee$, defines matroids $M(L_1)$ and $M(L_2)$, which form a matroid quotient $M(L_1) \twoheadleftarrow M(L_2)$.
\end{eg}

\begin{eg}[Canonical matroid quotients] Just as any linear space $L$ has two canonical linear maps, the identity $L \to L$ and the zero map $L \to 0$, any matroid $M$ has two canonical matroid quotients, the identity $M\twoheadrightarrow M$ and the trivial quotient $M \twoheadrightarrow U_{0,n}$.
\end{eg}

A matroid quotient $M_1\twoheadleftarrow M_2$ is an \textbf{elementary quotient} if $r(M_2) - r(M_1) = 1$.  Every matroid quotient $M_1\twoheadleftarrow M_2$ can be realized as a composition of a series of elementary quotients.  A canonical one is given by the \textbf{Higgs factorization} $M_1 = M^{(r_2 - r_1)}\twoheadleftarrow \cdots\twoheadleftarrow M^{(1)} \twoheadleftarrow M^{(0)} = M_2$, defined by $\mathcal B(M^{(i)}) = \{S \subseteq [n] \mid |S| = r(M_2)-i,\ S \textnormal{ spans $M_1$ and is independent in } M_2\}$.  The subsets $S\subseteq [n]$ that span $M_1$ and are independent in $M_2$ are called \textbf{pseudo-bases} of $(M_1, M_2)$. The set of pseudo-bases of $\MM=(M_1,M_2)$ is denoted by $p \mathcal{B}(\MM)$.
For a more on matroid quotients, we refer the reader to \cite[\S7.4]{Bry86} or \cite[\S7.3]{Oxl11}.

\begin{defn}
A \textbf{flag matroid} is a sequence of matroids $\MM = (M_1, \ldots, M_k)$\footnote{We remark that, unlike \cite{BGW03} but in agreement with \cite{EH20} and \cite{CDMS20}, we allow repetition of matroids in the sequence of matroids that constitute a flag matroid.} on a ground set $[n]$ such that $M_i \twoheadleftarrow M_{i+1}$ for all $i = 1, \ldots, k-1$.  The matroids $M_i$ are \textbf{constituents} of $\MM$, and the \textbf{rank} of $\MM$ is the sequence of ranks of its constituents $(r(M_1), \ldots, r(M_k))$.  The set of \textbf{bases} of $\MM$, denoted $\mathcal B(\MM)$, is the set of all $k$-flags of subsets $(B_1 \subseteq B_2 \subseteq \cdots \subseteq B_k)$ such that $B_i \in \mathcal B(M_i)$.
\end{defn}

\begin{eg}[Linear flag matroids] \label{eg:linearflag} A sequence of matroids $(M(L_1), \ldots, M(L_k))$ defined by a flag $\mathbf L$ of linear subspaces $L_1 \subseteq \cdots \subseteq L_k \subseteq \CC^n$ is a flag matroid.  We denote this flag matroid by $M(\mathbf L)$.  Flag matroids arising in this way are called linear (or realizable) flag matroids.
\end{eg}

For $\SS=(S_1,\ldots,S_k)$ a flag of subsets of $[n]$, we write $\be_{\SS} = \be_{S_1}+\cdots+\be_{S_k}$.  The \textbf{base polytope} $Q(\MM)$ of a flag matroid $\MM$ is the convex hull of $\{\be_{\BB} \mid \BB \in \mathcal B(\MM)\}$, whose vertices are in bijection with the bases of $\MM$.
The polytope $Q(\MM)$ is also the Minkowski sum of the base polytopes $Q(M_i)$ of the constituents of $\MM$.  The classical theorem of Gelfand, Goresky, MacPherson, and Serganova \cite{GGMS87} a characterizes base polytopes of matroids.  The analogue for flag matroids holds:

\begin{thm}\cite[Theorem 1.11.1]{BGW03} \label{thm:flagMatroidBasePolytope}
A lattice polytope $P \subset \RR^n$ is the base polytope of a rank $(r_1,\ldots,r_k)$ flag matroid on $[n]$ if and only if the following two conditions hold:
\begin{enumerate}
\item every vertex of $P$ is a $\mathfrak S_n$-permutation of $\be_{\{1,2,\ldots, r_1\}} + \cdots + \be_{\{1,2,\ldots, r_k\}}$, and
\item every edge of $P$ is parallel to $\be_i-\be_j$ for some $i,j \in [n]$.
\end{enumerate}
In particular, the normal fan of the base polytope $Q(\MM)$ of a flag matroid is a coarsening of the braid arrangement, which is the normal fan of the zonotope $\sum_{1\leq i < j \leq n} \operatorname{Conv}(\be_i, \be_j)$.
\end{thm}

Consequently, every face of a base polytope of a flag matroid is again a base polytope of a flag matroid.  The faces can be described explicitly.  For $u \in \RR^n$ and a polytope $Q\subset \RR^n$, let $Q^{u} := \{x\in Q \mid \langle x, u\rangle  = \max_{y\in Q} \langle y, u \rangle \}$ be the face maximizing in the direction of $u$.

\begin{prop}\label{prop:face}
Let $\MM = (M_1, \ldots, M_k)$ be a flag matroid on $[n]$ or rank $\br = (r_1, \ldots, r_k)$, and let $\SS = S_1 \subseteq \cdots \subseteq S_m$ be a flag of subsets of $[n]$.  Then $Q(\MM)^{\be_\SS}$ is the base polytope of a flag matroid whose $i$-th constituent (for $i = 1,\ldots, k$) is 
\[
M_i|S_1 \oplus M_i|S_2/S_1 \oplus \cdots \oplus M_i|S_m/S_{m-1} \oplus M_i/S_m.
\]
In other words, the bases of the flag matroid of $Q(\MM)^{\be_\SS}$ are bases $\BB = (B_1, \ldots, B_k)$ of $\MM$ such that $\operatorname{rk}_{M_i}(S_j) = |B_i \cap S_j|$ for all $1\leq i \leq k$ and $1\leq j \leq m$.
\end{prop}

\begin{proof}
Note that if $Q = \sum_{i=1}^k Q_i$ is a Minkowski sum of polytopes, then for any $u\in \RR^n$, the face $Q^u$ is the Minkowski sum $\sum_{i=1}^k Q_i^u$ of faces.  The proof of the proposition is thus reduced to the case of $\MM$ being a matroid $M$.  In this case, the statement is an immediate consequence of the greedy algorithm structure for matroids.
\end{proof}

\subsection{Torus-equivariant $K$-theory of flag varieties}\label{subsection:GKZ}
We will study combinatorial properties of flag matroids through the geometry of (partial) flag varieties and their (torus-equivariant) $K$-theory.  We point to \cite[\S8]{CDMS20} or \cite[\S2]{FS10} (and references therein) for a detailed exposition of equivariant $K$-theory of flag varieties.

\medskip
We begin by describing torus-equivariant $K$-theory and the method of localization.  Let $T = (\CC^*)^n$, and write $\ZZ[\mathbf t^\pm] := \ZZ[t_1^\pm, \ldots, t_n^\pm] = \ZZ[\ZZ^{[n]}]$ for the \textbf{character ring} of $T$.  Let $X$ be a smooth variety with a $T$-action, and let $\mathcal E$ be a ($T$-equivariant) vector bundle on $X$.  We write:
\begin{itemize}
\item $K^0(X)$ for the Grothendieck ring of vector bundles on $X$, which is isomorphic to the Grothendieck group of coherent sheaves $K_0(X)$ since $X$ is smooth,
\item $K^0_T(X)$ for the $T$-equivariant Grothendieck ring,
\item $[\mathcal E] \in K^0(X)$ for the $K$-class of $\mathcal E$ and $[\mathcal E]^T\in K^0_T(X)$ for its $T$-equivariant $K$-class,
\item $f_*$ for the (derived) pushforward map and $f^*$ for the pullback map of $K$-classes along a proper map $f: X\to X'$ of smooth varieties,
\item $\chi$ for the pushforward along the structure map $X \to \operatorname{Spec} \CC$, and
\item $\chi^T$ for the $T$-equivariant pushforward to $K_T^0(pt) = \ZZ[\mathbf t^\pm]$, the Lefschetz trace \cite[\S4]{Nie74}.
\end{itemize}

\medskip
We now assume that $X$ has finitely many $T$-fixed points, denoted $X^T$, and finitely many 1-dimensional $T$-orbits.  Moreover, we assume $X$ to be equivariantly formal and contracting, the precise definitions of which can be found in \cite[Remark 8.5 \& Definition 8.6]{CDMS20}.  Examples of such $X$ include flag varieties and smooth toric varieties.  By definition, for each $T$-fixed point $x\in X^T$, there is a $T$-invariant affine neighborhood $U_x \simeq \AA^{\dim X}$ whose characters $\{\lambda_1(x), \ldots, \lambda_{\dim X}(x)\}\subset \ZZ^n$ generate a pointed semigroup.  Fundamental results from the method of equivariant localization are collected in the following theorem.

\begin{thm}\label{thm:localization}
Let $X$ be a equivariantly formal and contracting smooth $T$-variety with finitely many $T$-fixed points and finitely many 1-dimensional $T$-orbits.  Then:
\begin{enumerate}
\item \cite[Corollary 5.12]{VV03}, \cite[Corollary A.5]{KR03} 
(cf.\ \cite[Theorem 3.2]{Nie74}, \cite[Theorem 2.7]{Tho83})
The restriction map
\[
K^0_T(X) \hookrightarrow K^0_T(X^T) \simeq (\ZZ[\mathbf t^\pm])^{X^T},\qquad \epsilon \mapsto \epsilon(\cdot)
\]
is injective.  Moreover, an element $\epsilon(\cdot) \in (\ZZ[\mathbf t^\pm])^{X^T}$ is in the image if and only if for every one-dimensional $T$-orbit in $X$ with boundary points $x,y\in X^T$ in the closure, the function $\epsilon(\cdot): X^T \to \ZZ[\mathbf t^\pm]$ satisfies
\[
\epsilon(x) \equiv \epsilon(y) \mod 1 - \mathbf t^\lambda
\]
 where $\lambda$ is the character of the action of $T$ on the one-dimensional orbit.

\item \cite[Theorem 2.6]{FS10}, \cite[Theorem 8.34]{MS05} Let $\mathcal E$ be a $T$-equivariant coherent sheaf on $X$, and let $x\in X^T$.  The image $[\mathcal E]^T(x)$ of $[\mathcal E]^T$ under the restriction $K_T^0(X) \to K_T^0(x) \simeq \ZZ[\mathbf t^\pm]$ is $\mathcal K(\mathcal E(U_x);\mathbf t)$ where
\[
\Hilb(\mathcal E(U_x)) := \frac{\mathcal K(\mathcal E(U_x);\mathbf t)}{\prod_{i=1}^{\dim X} (1-\mathbf t^{-\lambda_i(x)})}
\]
is the multigraded Hilbert series of the $\OO_X(U_x)$-module $\mathcal E(U_x)$ \cite[Theorem 8.20]{MS05}.

\item \cite[Theorem 5.11.7]{CG10} (cf.\ \cite[\S4]{Nie74}) Let $f: X \to Y$ be a proper $T$-equivariant map of equivariantly formal, contracting, and smooth $T$-varieties, and let $\alpha \in K_T^0(X)$, $\beta \in K_T^0(Y)$.  Then we have
\[
(f^*\beta)(x) = \beta(f(x)) \textnormal{ for every } x\in X^T, \textnormal{ and}
\]
\[
(f_*\alpha)(y) = \left(\prod_{i=1}^{\dim Y}(1-\mathbf t^{-\lambda_i(y)})\right) \left(\sum_{x \in X^T \cap f^{-1}(y)} \frac{\alpha(x)}{\prod_{i=1}^{\dim X}(1-\mathbf t^{-\lambda_i(x)})}\right) \textnormal{ for every } y\in Y^T.
\]
\end{enumerate}
\end{thm}

We now specialize our discussion of $K$-theory to flag varieties.
For a sequence of non-negative integers $\br = (r_1, \ldots, r_k)$ such that $0< r_1 \leq \cdots \leq r_k < n$, denote by $Fl(\br;n)$ the \textbf{flag variety}
\[
Fl(\br;n) := \{{\mathbf L} = (L_1 \subseteq \cdots \subseteq L_k \subseteq \CC^n) \textnormal{ linear subspaces with } \dim L_i = r_i \ \forall 1\leq i\leq k\}.
\]
For each $i = 1, \ldots, k$, we have the tautological sequence of vector bundles on $Fl(\br;n)$
\[
0\to \mathcal S_i \to \CC^n \to \mathcal Q_i \to 0
\]
where $\mathcal S_i$ is the ($i$-th) universal subbundle.  It is a vector bundle whose fiber at a point ${\mathbf L} \in Fl(\br;n)$ is the subspace $L_i$.  For $\mathbf a = (a_1, \ldots, a_k) \in \ZZ^k$ we denote by $\OO(\mathbf a)$ the line bundle $\bigotimes_{i=1}^k (\det \mathcal S_i^\vee)^{\otimes a_i}$, and by $\OO(\mathbf 1)$ the line bundle $\OO(1,1,\ldots, 1)$ on $Fl(\br;n)$.  The torus $T := (\CC^*)^n$ acts on $Fl(\br;n)$ by its action on $\CC^n$ where $(t_1, \ldots, t_n) \cdot (x_1, \ldots, x_n) = (t_1^{-1}x_1, \ldots, t_n^{-1}x_n)$.  With this $T$-action, a flag variety is a equivariantly formal and contracting space with the following structure:
\begin{itemize}
\item The $T$-fixed points $x_{\SS}$ of $Fl(\br;n)$ are flags of coordinate subspaces, which are in bijection with flags of subsets $S_1 \subseteq \cdots \subseteq S_k\subseteq [n]$ with $|S_i| = r_i$ for all $i = 1, \ldots, k$.
\item For a flag $\SS$, denote by $Ex(\SS)$ the set of $(i,j) \in [n]\times [n]$ such that $i\in S_\ell$ and $j\notin S_\ell$ for some $1\leq \ell \leq k$.  Then the set of characters of the $T$-neighborhood $U_{\SS}$ of $x_{\SS}$ is $\{\be_i - \be_j \mid (i,j)\in Ex(\SS)\}$.
\end{itemize}
The sign-convention we have adopted for the action of $T$ ensures that $T$ acts on the sections of $\mathcal S_i^\vee$ by positive characters.  For instance, we have $[\mathcal S_i^\vee]^T(x_\SS) = \sum_{j\in S_i} t_j$ and $[\mathcal Q_i^\vee]^T(x_\SS) = \sum_{j\in [n]\setminus S_i} t_j$, and moreover $[\bigwedge^p \mathcal S_i^\vee]^T(x_\SS) = \displaystyle \sum_{\tiny \begin{matrix} A \subseteq S_i \\ |A| = p \end{matrix}} \bt^{\be_A}$ and $ [\bigwedge^p \mathcal S_i]^T(x_\SS) = \displaystyle\sum_{\tiny \begin{matrix} A \subseteq S_i \\ |A| = p \end{matrix}} \bt^{-\be_A}$ (likewise for $\bigwedge^q \mathcal Q_i^\vee, \bigwedge^q \mathcal Q_i$).

\subsection{$K$-class of a flag matroid}\label{subsection:flagmatKclass}
Flag matroids enter into the $K$-theory of flag varieties as $T$-equivariant $K$-classes as follows.  Let $\MM$ be a flag matroid of rank $\br$ on a ground set $[n]$.  For $\BB$ a basis of $\MM$, define a polyhedral cone $\Cone_{\BB}(\MM):= \Cone(Q(\MM) - \be_{\BB}) \subset \RR^n$, also known as the tangent cone of $Q(M)$ at the vertex $\be_\BB$, and let $\Hilb_{\BB}(\MM)$ be the multigraded Hilbert series of $\CC[\mathbf t^\lambda \mid \lambda \in \Cone_{\BB}(\MM) \cap \ZZ^n]$ 
(see \cite[Theorem 8.20]{MS05}).

\begin{defn}\cite[Definition 8.19]{CDMS20}
Let $\MM$ be a flag matroid of rank $\br$ on a ground set $[n]$.  Then define $y(\MM)^T(\cdot)\in K^0_T(Fl(\br;n)^T)$ by
\[
y(\MM)^T(x_{\SS}) := \begin{cases}
\displaystyle \Hilb_{\SS}(\MM) \cdot \prod_{(i,j) \in Ex(\SS)} (1- t_i^{-1}t_j) & \textnormal{if $\SS$ a basis of $\MM$}\\
0 & \textnormal{otherwise}.
\end{cases}
\]
\end{defn}

By combining  \Cref{thm:localization}.(1) and \Cref{thm:flagMatroidBasePolytope}, one observes that $y(\MM)^T$ can be considered as a class in $K_T^0(Fl(\br;n))$ \cite[Proposition 8.20]{CDMS20}.  We will write $y(\MM)$ for the underlying non-equivariant $K$-class.  The geometric motivation for this $K$-class constitutes the remark below.

\begin{rem} Recall from \Cref{eg:linearflag} that a point ${\mathbf L} \in Fl(\br;n)$ defines a flag matroid $\MM := M({\mathbf L})$ of rank $\br$.  One observes that the torus-orbit closure $\overline{T\cdot {\mathbf L}}$ is isomorphic to the toric variety of the base polytope $Q(\MM)$, and then by applying \Cref{thm:localization}.(2) one shows that the class $[\OO_{\overline{T\cdot {\mathbf L}}}]^T \in K_T^0(Fl(\br;n))$ satisfies $[\OO_{\overline{T\cdot {\mathbf L}}}]^T(\cdot) = y(\MM)^T(\cdot)$.  See \cite[\S8.5]{CDMS20} for details.
\end{rem}

\begin{rem}\label{rem:valuative}
Let $\fP(\mathsf{FMat}_{\br;n})$ be a group generated by the indicator functions $\One(Q): \RR^n \to \RR$ of base polytopes $Q$ of rank $\br$ flag matroids on $[n]$.  A function $\varphi$ from the set of flag matroids of rank $\br$ on $[n]$ to an abelian group $A$ is \textbf{(strongly) valuative} if it factors through $\fP (\mathsf{FMat_{\br;n}})$.  As taking tangent cones and taking Hilbert series are valuative, it follows easily from the definition that the assignment $\MM \mapsto y(\MM)$ is valuative.
\end{rem}

When $\br = (r)$ (that is, we are concerned with the Grassmannian $Gr(r;n)$ and hence matroids of rank $r$ on $[n]$), invariants of a matroid $M$ built from $y(M)$ were explored in \cite{Spe09} and \cite{FS12} as follows.  To avoid confusion we write $\PP^{n-1}$ for $Gr(1;n)$ and $(\PP^{n-1})^\vee$ for $Gr(n-1;n)$.
Recall the diagram:
\begin{equation}\label{eqn:fouriermukai1}
\xymatrix{
& &Fl(1,r,n-1;n) \ar[ld]_{\pi_\br} \ar[rd]^{\pi_{(n-1)1}} &\\
&Gr(r;n) & & (\PP^{n-1})^\vee \times \PP^{n-1}.
}
\end{equation}
Let $\alpha$ be the $K$-class of the structure sheaf of a hyperplane in $(\PP^{n-1})^\vee$ and $\beta$ the likewise $K$-class from $\PP^{n-1}$.  We remark that our notation of $\alpha,\beta$ is flipped from the notation in \cite{FS12}\footnote{In \cite{FS12}, the authors consider $\pi_{1(n-1)}: Fl(1,r,n-1;n) \to \PP^{n-1}\times(\PP^{n-1})^\vee$, and set $\alpha$ and $\beta$ as the $K$-classes of the structure sheaves of hyperplanes from $\PP^{n-1}$ and $(\PP^{n-1})^\vee$ (respectively).  Our flipped naming of $\alpha,\beta$ is to remedy a minor error in the proof of \cite[Lemma 4.1]{FS12} (bottom three lines on pg.\ 2709), which accidentally flips the correspondence of $\alpha,\beta$ to appropriate $K$-classes.}.  Recall that $K^0((\PP^{n-1})^\vee \times \PP^{n-1}) \simeq \QQ[\alpha,\beta]/( \alpha^n, \beta^n)$.

\begin{thm}\label{thm:tuttematroid} \cite[Theorem 5.1]{FS12}
Let $M$ be a matroid of rank $r$ on $[n]$, and let $T_M(x,y)$ be its Tutte polynomial.  Then we have
\[
T_M(\alpha,\beta) = (\pi_{(n-1)1})_* \pi_r^* \Big( y (M) \cdot [\OO(1)]\Big).
\]
\end{thm}

\medskip
We will generalize this $K$-theoretic formulation of Tutte polynomials of matroids to flag matroids in two different ways in subsequent sections.  In both cases, similarly to \Cref{thm:tuttematroid}, the Tutte polynomials of flag matroids are formulated via diagrams like \eqref{eqn:fouriermukai1}, which we introduce in the next section.

\section{Two diagrams and a fundamental computation}\label{section:fundcomp}

The main goal of this section is to prove \Cref{prop:fundcomp}, which relates a pushforward of a pullback of $K$-classes to Euler characteristics of certain associated sheaves.  As this section is closely adapted from \cite[\S4]{FS12}, we only give sketches of proofs, save for the modified parts.

\medskip
Let $\br = (r_1, \ldots, r_k)$ be a sequence of non-negative integers.  For each $i = 1, \ldots, k$, recall that we have tautological bundles $\mathcal S_i$ and $\mathcal Q_i$ on $Fl(\br;n)$ fitting into the short exact sequences
\begin{equation}\label{eqn:SES}
0\to \mathcal S_i \to \CC^n \to \mathcal Q_i \to 0.
\end{equation}
For two vector bundles $\mathcal E, \mathcal F$ on $X = Fl(\br;n)$, we write $\pi: \operatorname{BiProj}(\mathcal E,\mathcal F)\to X$ for the bi-projectivization of the direct sum $\mathcal E \oplus \mathcal F$.  That is, $\operatorname{BiProj}(\mathcal E, \mathcal F) := \operatorname{Proj}(\operatorname{Sym}^\bullet \mathcal E )\times_X \operatorname{Proj}(\operatorname{Sym}^\bullet \mathcal F)$, so that for each point $x\in X$, the fiber $\pi^{-1}(x)$ is $\PP(\mathcal E_x^\vee) \times \PP(\mathcal F_x^\vee)$.  We consider the following two distinguished cases; note that the two cases are identical when $k=1$ (i.e.\ when $Fl(\br;n)$ is a Grassmannian  $Gr(r;n)$).

\begin{itemize}
\item $\operatorname{BiProj}(\mathcal S_1^\vee, \mathcal Q_k) \simeq Fl(1,\br,n-1;n)$.  In this case, we have maps:
\[ \label{construction1}\tag{$Fl$}
\xymatrix{
& &Fl(1,\br,n-1;n) \ar[ld]_{\pi_\br} \ar[rd]^{\pi_{(n-1)1}} &\\
&Fl(\br;n) & & (\PP^{n-1})^\vee \times \PP^{n-1}
}
\]
where $\pi_\br$ and $\pi_{(n-1)1}$ are given by forgetting the linear spaces of appropriate dimensions.
\item $\operatorname{BiProj}(\mathcal S_k^\vee, \mathcal Q_1) \simeq \widetilde{Fl}(1,\br,n-1;n)$ where $\widetilde{Fl}(1,\br,n-1;n)$ is a variety
\[ 
\widetilde{Fl}(1,\br,n-1;n) := \left\{\left.\begin{matrix}\textnormal{linear subspaces}\\ (\ell, L_1, \ldots, L_k, H)\end{matrix} \ \right| \ \begin{matrix}\textnormal{$\dim \ell = 1$, $\dim H = n-1$, $(L_1, \ldots, L_k) \in Fl(\br;n)$,} \\ \textnormal{ and $\ell \subseteq L_k$ and $L_1 \subseteq H$} \end{matrix}\right\}.
\]
In this case, we also have maps:
\[ \label{construction2}\tag{$\widetilde{Fl}$}
\xymatrix{
& &\widetilde{Fl}(1,\br,n-1;n) \ar[ld]_{\widetilde\pi_\br} \ar[rd]^{\widetilde\pi_{(n-1)1}} &\\
&Fl(\br;n) & & (\PP^{n-1})^\vee \times \PP^{n-1}
}
\]
where $\widetilde\pi_\br$ and $\widetilde\pi_{(n-1)1}$ are given by forgetting the linear spaces of appropriate dimensions.
\end{itemize}

\medskip
As before, let $\alpha = [\OO_{H_1}]$ be the $K$-class of the structure sheaf of a hyperplane in $(\PP^{n-1})^\vee$ and $\beta = [\OO_{H_2}]$ the likewise $K$-class from $\PP^{n-1}$.  The main statement of this section is as follows.

\begin{prop}\label{prop:fundcomp}
Let $\epsilon \in K^0(Fl(\br;n))$.  With $u$ and $v$ as formal variables, define polynomials
\[
R_\epsilon(u,v) := \sum_{p,q} \chi \Big( \epsilon\cdot [\bigwedge^p \mathcal S_k][\bigwedge^q \mathcal Q_1^\vee] \Big)u^pv^q \quad\textnormal{and}\quad \widetilde R_\epsilon(u,v) := \sum_{p,q} \chi \Big( \epsilon\cdot [\bigwedge^p \mathcal S_1][\bigwedge^q \mathcal Q_k^\vee] \Big)u^pv^q.
\]
Then we have the following identities in $K^0((\PP^{n-1})^\vee \times \PP^{n-1})$.
\[
R_\epsilon(\alpha-1,\beta-1) = (\pi_{(n-1)1})_*\pi_\br^*( \epsilon) \quad\textnormal{and}\quad \widetilde R_\epsilon(\alpha-1,\beta-1) = (\widetilde\pi_{(n-1)1})_*\widetilde\pi_\br^*( \epsilon). 
\]
\end{prop}

When $k=1$, i.e.\ $Fl(\br;n)$ is a Grassmannian, \Cref{prop:fundcomp} reduces to \cite[Lemma 4.1]{FS12}.  We remark that, just as in \cite{FS12}, \Cref{prop:fundcomp} is an identity in the \emph{non}-equivariant $K$-theory.
The proof of \Cref{prop:fundcomp} is a minor modification of the proof of \cite[Lemma 4.1]{FS12}.  Here, as a lemma, we separate out and also fix a minor error in the part of the proof in \cite{FS12} that needs modification.

\begin{lem}\label{lem:fundcomp} Denote $\eta_1 := (1-\alpha)^{-1} = [\OO(1,0)]$ and $\eta_2 := (1-\beta)^{-1} = [\OO(0,1)]$, and let $t$ be a formal variable. Then the following identities hold in $K^0(Fl(\br;n))[[t]]$.
\[
\sum_p [\bigwedge^p \mathcal S_k]t^p = (1+t)^n (\pi_\br)_* \pi_{(n-1)1}^* \Big(\frac{1}{1+t\eta_1}\Big) \textnormal{ and }
\sum_q [\bigwedge^q \mathcal Q_1^\vee]t^q = (1+t)^n (\pi_\br)_* \pi_{(n-1)1}^* \Big(\frac{1}{1+t\eta_2}\Big).
\]
And likewise,
\[
\sum_p [\bigwedge^p \mathcal S_1]t^p = (1+t)^n (\widetilde\pi_\br)_* \widetilde\pi_{(n-1)1}^* \Big(\frac{1}{1+t\eta_1}\Big) \textnormal{ and }
\sum_q [\bigwedge^q \mathcal Q_k^\vee]t^q = (1+t)^n (\widetilde\pi_\br)_* \widetilde\pi_{(n-1)1}^* \Big(\frac{1}{1+t\eta_2}\Big).
\]
\end{lem}

\begin{proof}
For each $i = 1, \ldots, k$ note that
\begin{equation}\label{eqn:lem1}
\Big(\sum_\ell [\bigwedge^\ell \mathcal S_i]t^\ell\Big) \Big(\sum_m [\bigwedge^m \mathcal Q_i]t^m \Big) = (1+t)^n,
\end{equation}
which follows from the short exact sequence \eqref{eqn:SES} and \cite[A2.2.(c)]{Eis95}.  We also have an identity
\begin{equation}\label{eqn:lem2}
\Big(\sum_\ell [\bigwedge^\ell \mathcal S_i] t^\ell \Big) \Big(\sum_m (-1)^m[\operatorname{Sym}^m \mathcal S_i]t^m \Big) = 1
\end{equation}
and likewise identities for $\mathcal Q_i$ and the duals $\mathcal S_i^\vee, \mathcal Q_i^\vee$, which follow from the exactness of the Koszul complex \cite[A2.6.1]{Eis95}.  Now, we note by \cite[Exercise III.8.4]{Har77} that
\begin{equation}\label{eqn:lem3}
\begin{split}
&(\pi_\br)_* \pi_{(n-1)1}^* (\eta_2^\ell \eta_1^m) = [\operatorname{Sym}^\ell \mathcal S_1^\vee \otimes \operatorname{Sym}^m \mathcal Q_k]\textnormal{ and}\\
&(\widetilde\pi_\br)_* \widetilde\pi_{(n-1)1}^* (\eta_2^\ell \eta_1^m) = [\operatorname{Sym}^\ell \mathcal S_k^\vee \otimes \operatorname{Sym}^m \mathcal Q_1].
\end{split}
\end{equation}
Combining \eqref{eqn:lem1}, \eqref{eqn:lem2}, and \eqref{eqn:lem3} then yields the desired identities.
\end{proof}

\begin{proof}[Sketch of proof of \Cref{prop:fundcomp}]
One combines \Cref{lem:fundcomp} with the projection formula for $K$-theory \cite[\S15.1]{Ful98}.  Then by expanding the power series in $u$ and $v$, which is in fact a finite sum, comparing coefficients yields the desired identity.  See the proof in \cite{FS12} for details.
\end{proof}

\section{Summations of lattice point generating functions}\label{section:latticegenfct}

The method of equivariant localization \S\ref{subsection:GKZ}, aided by \Cref{prop:fundcomp}, will reduce our $K$-theoretic computations to summations of lattice point generating functions.  Here we collect some useful results concerning summations of lattice point generating functions arising from polyhedra, along with variants that are suitable for our purposes.   Our main novel contribution is \Cref{thm:sumOfCones}, which is a useful variant of the method of flipping cones.  The reader may see \Cref{eg:illustration} for illustrations of main theorems here.

\subsection{Brion's formula} Here we review the results in \cite{Bri88, Ish90}.  For a subset $S\subset \RR^n$ , denote by $\One(S): \ZZ^n \to \QQ$ its indicator function sending $x \mapsto 1$ if $x\in S$ and 0 otherwise.  Let $\fP_n$ be a vector space of $\QQ$-valued functions on $\ZZ^n$ generated by $\{\One(P) \mid P\subset \RR^n \textnormal{ lattice polyhedra}\}$. It follows from the Brianchon-Gram formula \cite{brianchon, gram, shephard} that $\fP_n$ is generated by indicator functions of cones, and by triangulating one conludes that $\fP_n$ is generated by indicator functions of smooth cones.

We will often consider elements of $\fP_n$ as elements in the power series ring $\QQ[[t_1^\pm, \ldots, t_n^\pm]]$ by identifying $\One(P)$ with $\sum_{\lambda\in P \cap \ZZ^n} \mathbf t^\lambda$.  The following fundamental theorem concerns convergence of these power series to a rational function.

\begin{thm}\cite[Theorem 1.2]{Ish90}\footnote{Fink and Speyer in \cite{FS12} and Postnikov in \cite{Pos09} cite \cite{KP92}, whereas Ishida in \cite{Ish90} writes that the theorem is originally due to Brion.} Consider $\fP_n$ as a $\QQ[t_1^\pm, \ldots, t_n^\pm]$-submodule of $\QQ[[t_1^\pm, \ldots, t_n^\pm]]$, and let $\QQ(t_1, \ldots, t_n)$ be the fraction field.  There exists a unique $\QQ[t_1^\pm, \ldots, t_n^\pm]$-linear map
\[
\Hilb: \fP_n \to \QQ(t_1, \ldots, t_n)
\]
such that if $C = \operatorname{Cone}(v_1, \ldots, v_k) \subset \RR^n$ is a smooth cone with primitive ray generators $v_1, \ldots, v_k \in \ZZ^n$ then $\Hilb(\One(C)) = \prod_{i=1}^k \frac{1}{1-\mathbf t^{v_i}}$.
\end{thm}

Two remarks about the above linear map $\Hilb$ follow:
\begin{enumerate}
\item The notation $\Hilb$ agrees with our previous notion of Hilbert series:  when $C$ is a pointed rational polyhedral cone, not necessarily smooth, $\Hilb(\One(C))$ equals the multigraded Hilbert series of $\CC[\mathbf t^\lambda \mid \lambda \in C\cap \ZZ^n]$ in the sense of \cite[Theorem 8.20]{MS05}.
\item If $P$ is a lattice polyhedron with a non-trivial lineality space, then $\Hilb(\One(P)) = 0$.
\end{enumerate}

For $P$ a lattice polyhedron, we will often by abuse of notation write $\Hilb(P)$ for $\Hilb(\One(P))$.
An important result on rational generating functions for cones is the formula of Brion \cite{Bri88}, which was slightly generalized in \cite{Ish90}.  Here we will only need the following special case of \cite[Theorem 2.3]{Ish90}.

\begin{thm} \label{thm:generalbrion}
Let $P \subset \RR^n$ be a lattice polyhedron with a nonempty set of vertices (so $P$ has no lineality space), and let $C(P)$ be its recession cone.  For every vertex $v$ of $P$, write $C_v$ for $\Cone(P-v)$.  Then we have
\[
\Hilb(P) = \sum_{v\in \operatorname{Vert}(P)} \Hilb(C_v + v) \qquad \textnormal{ and }\qquad \Hilb(C(P)) = \sum_{v\in \operatorname{Vert}(P)} \Hilb(C_v).
\]
\end{thm}

\subsection{Lawrence-Varchenko formula (flipping cones) and variants}\label{subsection:flipping}

Here we review the method of flipping cones \cite[\S6]{FS10}, \cite[(11)]{BHS09}.  Our contribution is a generalization \Cref{thm:sumOfCones}, which will serve as a key technical tool in subsequent sections.

\medskip
Let $\bz \in \RR^n$. For every $a \in \RR$, we will denote the hyperplane $\{x \in \RR^n|\langle \bz,x \rangle = a\}$ by $H_{\bz= a}$ and the half-space $\{x \in \RR^n|\langle \bz,x \rangle \geq a\}$ by $H_{\bz\geq a}$.  For an element $f\in \mathcal P_n$, by considering $f$ as an element of $\QQ[[t_1^\pm, \ldots, t_n^\pm]]$ we write $f|_{H_{\bz = a}}$ for the sum of terms $c t^w$ in $f$ such that $\langle w, \bz \rangle = a$.

\begin{defn}
A polyhedron $P \subset \RR^n$ is \textbf{$\bz$-pointed} if $P \subseteq H_{\bz \geq a}$ for some $a \in \RR$. Let $\fP_n^{\bz}$ be the $\QQ$-vector space generated by $\bz$-pointed elements in $\fP_n$.
\end{defn}
 
We note the following useful observation:  Let $P\subset \RR^n$ be a polyhedron with vertices $\operatorname{Vert}(P)$, and as before let $C_v := \operatorname{Cone}(P-v)$ for $v\in \operatorname{Vert}(P)$.  Then, for $\bz\in \RR^n$, the cone $C_v$ is $\bz$-pointed if and only if $v$ is a vertex of the face $P^{-\bz}$ of $P$ minimizing in the $\bz$ direction.

\smallskip
If $f\in \fP_n^\bz$, then one can compute $\Hilb(f)$ "slice-by-slice" in the following sense.

\begin{lem} \label{lem:flipInj}
    Let $f,g \in \fP_n^{\bz}$ and suppose that $\Hilb(f)=\Hilb(g)$. Then for every $a \in \RR$, it holds that $\Hilb(\left.f\right|_{H_{\bz= a}})=\Hilb(\left.g\right|_{H_{\bz= a}})$.
\end{lem}
\begin{proof}
Write $b=f-g$, and suppose by contradiction that there is an $a \in \RR$ with $\Hilb(\left.b\right|_{H_{\bz= a}}) \neq 0$. Since $b \in \fP_n^{\bz}$, there is a minimal such $a$, which we denote by $a_0$. Writing $b = \sum_i p_i \One(C_i)$ with $C_i$ smooth cones and $p_i$ Laurent polynomials, we define a nonzero Laurent polynomial $q(\bt) = \sum_{e\in \ZZ^n}{\lambda_e \boldsymbol{t}^e}$ by $q(\bt) := \prod_i \prod_{j_i} (1 - \bt^{j_i})$ where $j_i$ ranges over the primitive rays of $C_i$.
By construction $q \cdot b$ has finite support, i.e.\ is a Laurent polynomial, and $\Hilb(q \cdot b)=q \Hilb(b)=0$. Hence, we have $q \cdot b=0$. Let $c = \min\{\langle \bz,e \rangle|\lambda_e \neq 0\}$, and let $q_0 = \sum_{e:\langle \bz,e \rangle=c}{\lambda_e \boldsymbol{t}^e}$. Then $0=\Hilb(\left.(q\cdot b)\right|_{H_{\bz= a_0+c}})=q_0\Hilb(b|_{H_{\bz= a_0}}) \neq 0$, a contradiction.
\end{proof}

Suppose that $\bz = (\zeta_1, \ldots, \zeta_n)$ is chosen such that the $\zeta_i$'s are $\QQ$-linearly independent, in which case we say ``$\bz$ is \textbf{irrational}.''
Then for any $a \in \RR$, the intersection $H_{\bz=a} \cap \ZZ^n$ consists of at most one point. In this case Lemma \ref{lem:flipInj} reduces to saying $\Hilb: \fP^{\bz}_n \to \QQ(t_1,\ldots,t_n)$ is injective, recovering \cite[Lemma 6.3]{FS10}.
We next recall the notion of cone flips. We begin with a lemma for their existence.

\begin{lem}{\cite[Lemma 6]{Haa05}, \cite[Lemma 2.1]{FS12}} \label{lem:coneFlip}
Assume $\bz$ is irrational.
For every $f\in \fP_n$, there is a unique $f^{\bz} \in \fP_n^{\bz}$ such that $\Hilb(f)=\Hilb(f^{\bz})$. The map $f \mapsto f^{\bz}$ is linear.
\end{lem}

The map $(\cdot)^{\bz}$ in the lemma is described explicitly as follows.  Let $C \subseteq \RR^n$ be a rational simplicial cone
\[
C =  \textstyle \{w+\sum_{i=0}^{n-1}{a_i}v_i \mid a_i \geq 0 \text { for all } i \in [n]\}.
\]
Then the image $C^{\bz} \in \fP_n^{\bz}$ under the map of Lemma \ref{lem:coneFlip} is given by
\begin{equation} \label{eq:coneFlip}
C^{\bz} = (-1)^\ell \One\left(\Big\{w+\sum_{i=0}^{n-1}{a_i}v_i\medspace \begin{array}{|c} a_i \geq 0 \text { for all } i \text{ with } \langle \bz, v_i \rangle >  0, \\ \text{ and }a_i < 0 \text { for all } i \text{ with } \langle \bz, v_i \rangle <  0 \end{array}\Big\}\right),
\end{equation}
where $\ell$ is the number of rays $v_i$ for which $\langle \bz, v_i \rangle <  0$. 
We will refer to $C^{\bz}$ as the \emph{cone flip} of $C$ in direction $\bz$.
For a general pointed rational cone $C$, one defines the flipped cone $C^{\bz} \in \fP_n^{\bz}$ by triangulating the cone\footnote{We remark that calling $C^{\bz}$ the "flipped cone" of $C$ is a slight abuse of terminology when $C$ is not simplicial, since $C^{\bz}$ is not necessarily the support function of a polyhedron.  It can be a genuine linear combination of some of those; see \cite[Remark 6.7]{FS10}.}.

\begin{rem}
The assumption that $\bz$ is irrational is essential for \Cref{lem:coneFlip}: if $\bz$ is not irrational then $\fP_n^{\bz}$ contains some lattice polyhedron $P$ with a non-trivial lineality space, and $\Hilb(P)=0=\Hilb(0)$, contradicting uniqueness.
\end{rem}

Now, suppose we are given an expression over a finite index set $\Lambda$
\begin{equation} \label{eqn:sumOfCones}
    \varphi=\sum_{\lambda\in \Lambda} {a_\lambda\Hilb(C_\lambda)} \in \QQ(t_1,\ldots,t_n),
\end{equation}
where the $C_\lambda$ are pointed cones with vertices not necessarily at the origin and $a_\lambda \in \QQ$ are scalars.  Suppose we know that $\varphi \in \QQ(t_1, \ldots, t_n)$ is in fact a Laurent polynomial, for example, because $\varphi$ arose from a computation in $T$-equivariant $K$-theory.  Then we can use cone-flipping to get partial information about the coefficients of $\varphi$.  The following proposition is our "cone-flipping in slices" technique which will be used repeatedly in later sections.

\begin{thm}\label{thm:sumOfCones}
Suppose $\varphi=\sum_\lambda{a_\lambda\Hilb(C_\lambda)}$ is a Laurent polynomial, i.e.\ $\varphi \in \QQ[t_1^{\pm},\ldots,t_n^{\pm}]$, and let $P$ be the convex hull of the vertices of the $C_\lambda$.  For $\bz\in \RR^n$, not necessarily irrational, and $b\in \RR$, suppose that every cone $C_\lambda$ whose vertex $w_\lambda$ satisfies $\langle \bz, w_\lambda \rangle < b$ is $\bz$-pointed.  Then
\[
\varphi|_{H_{\bz=b}}=\sum_{C_\lambda \in \fP_n^{\bz}}{a_\lambda \Hilb(C_\lambda \cap H_{\bz=b})}.
\]
In particular, if $P \cap H_{\bz = b}$ is the face $P^{-\bz}$ of $P$ minimizing in the $\bz$ direction, then
\[
\varphi|_{H_{\bz = b}} = \sum_{\tiny \begin{matrix} \textnormal{$\bz$-pointed $C_\lambda$ whose} \\ \textnormal{vertex $w_\lambda$ is on $P^{-\bz}$} \end{matrix}} a_\lambda \Hilb(C_\lambda \cap H_{\bz = b}).
\]
\end{thm}

We note two useful immediate consequences of \Cref{thm:sumOfCones} in the following corollary, of which the second statement appeared previously in \cite{FS10}.

\begin{cor}\label{cor:vertexcase}
With assumptions as in \Cref{thm:sumOfCones}, one has the following:
\begin{enumerate}[label=(\alph*)]
    \item \label{vertexcase:1} If $H_{\bz=b}\cap P=\{w\}$ is a vertex of $P$, the coefficient of $\bt^w$ in $\varphi$ is equal to $\sum{a_\lambda}$, where the sum is over all $\lambda$ for which $C_\lambda \in \fP_n^{\bz}$ and the vertex of $C_\lambda$ is at $w$.
    \item \label{vertexcase:2} \cite[Corollary 6.9]{FS10} The Newton polytope $\operatorname{Newt}(\varphi)$ of $\varphi$ is contained in $P$.
\end{enumerate}
\end{cor}

\begin{proof}
The first statement is a special case of \Cref{thm:sumOfCones}.  For the second statement, observe that for any lattice point $v\in  \operatorname{Newt}(\varphi)$ and any $\bz\in \RR^n$, there must exist a cone $C_\lambda$ such that its vertex $w_\lambda$ satisfies $\langle \bz, w_\lambda\rangle \leq b$ where $b = \langle \bz, v\rangle$, since otherwise $\varphi|_{H_{\bz = b}} = 0$ by \Cref{thm:sumOfCones}.
\end{proof}

We prepare for the proof by noting a useful feature of the cone-flipping operation, starting with the following notion.

\begin{defn} \label{def:irrApprox}
Let $C$ be a pointed cone, and $\bz \in \RR^n$. We say that an irrational $\bz' \in \RR^n$ is an \textbf{irrational approximation} of $\bz$ with respect to $C$, if for every ray generator $v \in \RR^n$ of $C$ it holds that $\langle \bz, v \rangle >0 \implies \langle \bz', v \rangle >0$ and  that $\langle \bz, v \rangle <0 \implies \langle \bz', v \rangle <0$.
	
\end{defn}

Note that an irrational approximation of $\bz$ can always be obtained as a small perturbation of $\bz$.  The following is a minor generalization of \cite[Lemma 2.3]{FS12}, with almost identical proof, which we have included for completeness.

\begin{lem}\label{lem:flipinhalfspace}
	Let $\bz \in \RR^n$, 
	let $C$ be a pointed cone with vertex at $w$, and let $\bz' \in \RR^n$ be an irrational approximation of $\bz$.  Then its cone flip $C^{\bz'}$ is supported in the half space $\{x \mid \langle \bz,x \rangle \geq  \langle \bz,w \rangle\}$. 
	Furthermore, if $C$ is not contained in $\{x \mid \langle \bz,x \rangle \geq \langle \bz,w \rangle\}$, then $C^{\bz'}$ is supported in the open half space $\{x \mid \langle \bz,x \rangle > \langle \bz,w \rangle\}$; in particular $w \notin C^{\bz'}$.
\end{lem}
\begin{proof}
If $C$ is simplicial, the result follows immediately from the construction of cone flips \eqref{eq:coneFlip} and \Cref{def:irrApprox}. For general $C$, we can obtain the first statement by considering any triangulation of $C$. For the second one, choose a ray $v$ of $C$ such that $\langle \bz,v \rangle< 0$ and a triangulation of $C$ such that every interior cone contains $v$. Such a triangulation can for instance be constructed by triangulating the faces of $C$ that do not contain $v$, and then coning that triangulation from $v$.
Now $C=\sum_F(-1)^{\dim C - \dim F}\One(F)$ and $C^{\bz'}=\sum_F(-1)^{\dim C - \dim F}\One(F)^{\bz'}$, where the sum is over all interior cones of the triangulation.  The result now follows from the simplicial case. 
\end{proof}

\begin{proof}[Proof of \Cref{thm:sumOfCones}]
Since the summation defining $\varphi$ is over a finite collection of cones $\{C_\lambda\}_{\lambda\in \Lambda}$, there exists a $\bz' \in \RR$ which is an irrational approximation of $\bz$ with respect to every cone $C_\lambda$.
By assumption $\varphi=\Hilb(f)$, where $f \in \fP_n$ has finite support, in particular $f \in \fP_n^{\bz}$.
Hence, by \Cref{lem:flipInj}, $\varphi|_{H_{\bz=b}}=\Hilb(\sum{a_\lambda\One(C_\lambda^{\bz'} \cap {H_{\bz=b}})})$. 
If $C_\lambda \notin \fP_n^{\bz}$, then by assumption the vertex $w_\lambda$ of $C_\lambda$ satisfies $\langle \bz, w_\lambda \rangle \geq b $, and by \Cref{lem:flipinhalfspace} $C_\lambda^{\bz'}$ is supported on the open half-space $\{x \mid \langle \bz,x \rangle > b\}$, in particular $C_\lambda^{\bz'} \cap H_{\bz=b}=\emptyset$.
If $C_\lambda \in \fP_n^{\bz}$, then since $C_\lambda$ and $C_\lambda^{\bz'}$ are both in $\fP_n^{\bz}$, it follows from \Cref{lem:flipInj} that $\Hilb(C_\lambda^{\bz'} \cap H_{\bz=b})=\Hilb(C_\lambda \cap H_{\bz=b})$.
\end{proof}

\subsection{Flipping cones for base polytopes}
Let us now specialize our discussion of summing lattice point generating functions to ones arising from flag matroids.
For the rest of this section, let $\MM$ be a flag matroid of rank $\br=(r_1,\ldots,r_k)$ on a ground set $[n]$, whose constituent matroids have rank functions $\rk_1,\ldots,\rk_k$.   As before, for a basis $\BB$ of $\MM$ let us write $\operatorname{Cone}_\BB(\MM) := \operatorname{Cone}(Q(\MM) - \be_\BB)$.

\medskip
Consider the expression below, which is a finite summation
\begin{equation} \label{eqn:sumOfMatroidCones}
    \varphi=\sum_{\lambda\in \Lambda}{a_\lambda \bt^{w_\lambda}\Hilb(\operatorname{Cone}_{\BBi}(\MM))},
\end{equation}
where $a_\lambda \in \QQ$, $w_\lambda \in \ZZ^n$, and $\BBi$ a basis of $\MM$.  We allow the same basis to occur several times in the sum.  Note that $\bt^{w_\lambda}\Hilb(\operatorname{Cone}_{\BBi}(\MM))=\Hilb(C_\lambda)$, where $C_\lambda$ is a cone with vertex at $w_\lambda$, so (\ref{eqn:sumOfMatroidCones}) is a special case of (\ref{eqn:sumOfCones}). As before, we assume that $\varphi \in \QQ[t_1^{\pm},\ldots,t_n^{\pm}]$, i.e.\ $\varphi$ is a Laurent polynomial, and we write $P := \operatorname{Conv}(w_\lambda \mid \lambda \in \Lambda)$ for the convex hull of the $w_\lambda$.
We will assume that all $w_\lambda$ lie in $\ZZ_{\geq 0}^n$, and that there exists a $c \in \ZZ_{\geq 0}$ such that the sum of the entries of any $w_\lambda$ is equal to $c$. 
Let $\widetilde{P} := \operatorname{Conv}(\sigma \cdot w_\lambda \mid \sigma \in S_n,\ \lambda \in \Lambda)$ be the convex hull of all points in $\ZZ_{\geq 0}^n$ that are equal to one of the $w_\lambda$ up to permuting entries.

\medskip
The following theorem will be repeatedly applied in the next sections.

\begin{thm}\label{thm:sumOfMatroidCones}
Let $\varphi$, $P$, and $\widetilde P$ be as above, and let $v$ be a vertex of $\widetilde{P}$. Write $v=\be_{S_1}+\cdots+\be_{S_m}$, with $S_1\subseteq \ldots \subseteq S_m \subseteq [n]$. Fix a basis  $\BB=(B_1,\ldots,B_k)$ of $\MM$ such that $\be_\BB$ is a vertex of the face $Q(\MM)^{v}$ of $Q(\MM)$ maximizing the direction $v$, that is, a basis $\BB$ satisfying $|S_i \cap B_j|=\rk_{M_j}(S_i)$ for all $1\leq i\leq m$ and $1\leq j \leq k$ (\Cref{prop:face}).  Then the coefficient of $\bt^v$ in $\varphi \in \QQ[t_1^\pm, \ldots, t_n^\pm]$ is equal to the sum of all $a_\lambda$ for which $w_\lambda=v$ and $\BBi=\BB$.
\end{thm}

\begin{proof}
Since the Newton polytope of $\varphi$ is contained in $P$ (\Cref{cor:vertexcase}.\ref{vertexcase:2}), the result is true for $v \notin P$.
So, we now consider the case $v \in P$.  Let us write $v = (v_1, \ldots, v_n)$ and $\be_\BB = (b_1, \ldots, b_n)$.  By permuting the coordinates of $\mathbb{N}^n$, we may assume that $v_i\geq v_{i+1}$ for all $i \in [n]$, and that $b_i \geq b_{i+1}$ whenever $v_i=v_{i+1}$.

We first show that we may choose a $\bz' =(\zeta_1', \ldots, \zeta_n')\in \RR^n$ that satisfies the following properties:
\begin{enumerate}[label=(\roman*)]
    \item\label{item:bz} The vertex $\{v\}$ is the face of $\widetilde P$ maximizing in the $\bz'$ direction, and hence is the vertex of $P$ maximizing in the $\bz'$ direction.
    \item $\zeta_1' > \zeta_2' > \ldots > \zeta_n'$.
\end{enumerate}
To choose such a $\bz'$, start with any $\bz'$ satisfying \ref{item:bz}, which by perturbing the entries, we may assume to have all distinct entries. For any pair $i,j \in [n]$ such that $v_i > v_j$, we have $\zeta_i' > \zeta_j'$, since else we can swap $v_i$ and $v_j$ and obtain a vertex of $\widetilde{P}$ where $\bz'$ attains a larger value. For any collection $i, i+1, \ldots, j \in [n]$ such that $v_{i} = \cdots = v_{j}$, we may reorder the corresponding entries of $\bz'$ in decreasing order, since such a reordering does not change the value of $\langle \bz', v \rangle$.  This procedure produces the desired $\bz'$ since we had assumed $(v_1, \ldots, v_n)$ to be weakly decreasing.

We then claim that the vertex face of $Q(\MM)$ maximizing in the $\bz'$ direction is $\{\be_\BB\}$. Indeed, note that $\bz'$ is an interior point in the cone  
\[
\operatorname{Cone}(\be_1, \be_1+\be_2, \ldots, \be_1+\cdots + \be_{n-1}) + \RR\be_{[n]},
\]
of which the cone
\[
\operatorname{Cone}(\be_{S_1}, \be_{S_2}, \ldots, \be_{S_m}) + \RR\be_{[n]}
\]
is a face.  This face contains $v$ in its relative interior.  These two cones are cones in the braid arrangement, of which the normal fan of $Q(\MM)$ is a coarsening (\Cref{thm:flagMatroidBasePolytope}).  Thus, the vertex face of $Q(\MM)$ maximizing in the $\bz'$ direction is among the vertices of $Q(\MM)^{v}$, and our assumption $b_i \geq b_{i+1}$ for all $i = 1, \ldots, n$ such that $v_i = v_{i+1}$ ensures that $\be_\BB$ is indeed the one.  Now, applying \Cref{cor:vertexcase}.\ref{vertexcase:1} with $\bz = - \bz'$ gives the desired statement.
\end{proof}

\begin{eg}\label{eg:illustration}
We illustrate \Cref{thm:generalbrion}, \Cref{thm:sumOfCones}, and \Cref{thm:sumOfMatroidCones} in an example. Let $\MM$ be the flag matroid $(U_{1,3},U_{2,3})$. Its base polytope $Q(\MM)$ is drawn on the left in \Cref{fig:eg}. We arrange the six vertex cones of $Q(\MM)$ as on the right hand side of the figure,
getting a summation
\[
\begin{split}
\varphi =& t_1t_2\Hilb(\Cone_{(2,1,0)}(\MM)) + t_1t_3\Hilb(\Cone_{(2,0,1)}(\MM)) \qquad(\text{in blue})\\
& + t_2^2\Hilb(\Cone_{(1,2,0)}(\MM))
+t_3^2\Hilb(\Cone_{(1,0,2)}(\MM)) \qquad(\text{in green})\\
& + t_2^2\Hilb(\Cone_{(0,2,1)}(\MM)) + t_3^2\Hilb(\Cone_{(0,1,2)}(\MM)) \qquad(\text{in red})
\end{split}
\]

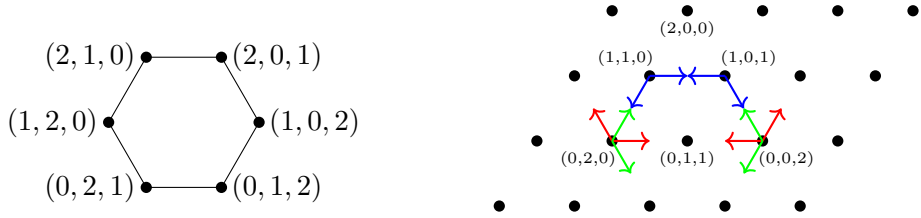
\begin{figure}[h]
\usetikzlibrary{calc}
\begin{tikzpicture}
         \foreach\i/\j in {0/0,1/0,-1/1,1/1,-1/2,0/2}{
     \node at (\i+0.5*\j,{\j*sqrt(3)*0.5}) [circle,fill,inner sep=1.5pt]{};
     }
    \draw[] (0,0)--(1,0)--(1.5,{0.5*sqrt(3)})--(1,{sqrt(3)})--(0,{sqrt(3)})--(-0.5,{0.5*sqrt(3)})--cycle;
    \node at (0,0) [left]{$(0,2,1)$};
    \node at (-0.5,{0.5*sqrt(3)}) [left]{$(1,2,0)$};
    \node at (0,{sqrt(3)}) [left]{$(2,1,0)$};
     \node at (1,0) [right]{$(0,1,2)$};
    \node at (1.5,{0.5*sqrt(3)}) [right]{$(1,0,2)$};
    \node at (1,{sqrt(3)}) [right]{$(2,0,1)$};
\end{tikzpicture}
\qquad\qquad
\begin{tikzpicture}
    \foreach \i in {-1, 0, 1, 2, 3}{\foreach \j in {-1,0,1,2}{
    \node at (\i+0.5*\j,{\j*sqrt(3)*0.5}) [circle,fill,inner sep=1.5pt]{};
    }}
    \draw[->,red,thick] (0,0)--(0.5,0);
    \draw[->,red,thick] (2,0)--(1.5,0);
    \draw[->,red,thick] (0,0)--(-0.25,{sqrt(3)*0.25});
    \draw[->,red,thick] (2,0)--(2.25,{sqrt(3)*0.25});
    \draw[->,green,thick] (0,0)--(0.25,{-sqrt(3)*0.25});
    \draw[->,green,thick] (2,0)--(1.75,{-sqrt(3)*0.25});
    \draw[->,green,thick] (0,0)--(0.25,{sqrt(3)*0.25});
    \draw[->,green,thick] (2,0)--(1.75,{sqrt(3)*0.25});
    \draw[->,blue,thick] (0.5,{sqrt(3)*0.5})--(0.25,{sqrt(3)*0.25});
    \draw[->,blue,thick] (0.5,{sqrt(3)*0.5})--(1,{sqrt(3)*0.5});
    \draw[->,blue,thick] (1.5,{sqrt(3)*0.5})--(1.75,{sqrt(3)*0.25});
    \draw[->,blue,thick] (1.5,{sqrt(3)*0.5})--(1,{sqrt(3)*0.5});
    \node at (0,0) [below left,xshift=0.5em]{$\scriptscriptstyle{(0,2,0)}$};
    \node at (1,0) [below]{$\scriptscriptstyle{(0,1,1)}$};
    \node at (2,0) [below right,xshift=-0.5em]{$\scriptscriptstyle{(0,0,2)}$};
    \node at (0.5,{sqrt(3)*0.5}) [above left,xshift=0.5em]{$\scriptscriptstyle{(1,1,0)}$};
    \node at (1.5,{sqrt(3)*0.5}) [above right,xshift=-0.5em]{$\scriptscriptstyle{(1,0,1)}$};
    \node at (1,{sqrt(3)}) [below]{$\scriptscriptstyle{(2,0,0)}$};
\end{tikzpicture}
\caption{Base polytope $Q(\MM)$ and translates of its vertex cones}
\label{fig:eg}
\end{figure}


By generalized Brion's theorem (second half of  \Cref{thm:generalbrion}), one may replace the two cones at $(0,2,0)$, one red and one green, by a single cone $\operatorname{Cone}(\be_1-\be_2, \be_3 - \be_2)$, and likewise for the two cones at $(0,0,2)$.  Then $\varphi$ is now a summation of the four vertex cones of the trapezoid $\operatorname{Conv}(2\be_2, 2\be_3, \be_1+\be_2, \be_1+\be_3)$, whose value by Brion's theorem (first half of \Cref{thm:generalbrion}) is 
\[
\varphi=t_1t_2+t_1t_3+t_2^2+t_2t_3+t_3^2.
\]

If we apply \Cref{thm:sumOfCones} with $\bz=\be_1$ and $b=0$, noting that the $\be_1$-pointed cones are $\Cone_{(0,2,1)}(\MM)$ and $\Cone_{(0,1,2)}(\MM)$, colored red in the figure, we find that, as expected,
\[
\varphi|_{H_{\be_1=0}}=\frac{t_2^2}{1-t_3t_2^{-1}}+\frac{t_3^2}{1-t_2t_3^{-1}}=\frac{t_2^3-t_3^3}{t^2-t_3}=t_2^2+t_2t_3+t_3^2.
\]

To apply \Cref{thm:sumOfMatroidCones}, we first note that the polytope $\widetilde{P}$ is the convex hull of $(2,0,0),(0,2,0),(0,0,2)$. Since there are no cones placed at $(2,0,0)$, \Cref{thm:sumOfMatroidCones} says that the coefficient of $t_1^2$ is equal to $0$. For the coefficient of $t_2^2$, we can take either $\BB=(1,2,0)$ or $\BB=(0,2,1)$; in both cases \Cref{thm:sumOfMatroidCones} tells us that the coefficient of $t_2^2$ is equal to $1$.
\end{eg}

In \Cref{eg:illustration}, because we had only one translate of each vertex cone, arranged in a suitable manner, we could apply Brion's theorem to compute $\varphi$. In subsequent sections, we will typically have several parallel translates of each vertex cone, where \Cref{thm:sumOfCones} or \Cref{thm:sumOfMatroidCones} will better suit our needs.

\section{The Las Vergnas Tutte polynomial of a matroid quotient}\label{section:LVT}

In \cite{LV75}, Las Vergnas introduced a Tutte polynomial of a matroid quotient as follows, and studied its properties in a series of subsequent works \cite{LV80, LV84, LV99, ELV04, LV07, LV13}.  The reader may find the survey \cite{LV80} particularly useful.

\begin{defn}\label{defn:LVT}
Let $\MM = (M_1, M_2)$ be a two-step flag matroid on a ground set $[n]$.  For $i = 1,2$ write $r_i$ for the rank of $M_i$ and $r_i(S)$ for the rank of $S\subseteq [n]$ in $M_i$.  The \textbf{Las Vergnas Tutte polynomial} of $\MM$ is
\begin{equation}\label{eqn:LVT}
LV\mathcal T_{\MM}(x,y,z) := \sum_{S \subseteq [n]} (x-1)^{r_1 - r_1(S)}(y-1)^{|S| - r_2(S)} z^{r_2 - r_2(S) - (r_1 - r_1(S))}
\end{equation}
\end{defn}

For the remainder of this section, we let $\MM$ be a two-step flag matroid, i.e.\ a matroid quotient $M_1 \twoheadleftarrow M_2$ on a ground set $[n]$.  We show in this section that $LV\mathcal T_{\MM}$ arises $K$-theoretically from $y(\MM)$.  We start by recalling the construction \eqref{construction2} of $\widetilde{Fl}(1,r_1,r_2,n-1;n)$ in \S\ref{section:fundcomp} with the maps
\[
\xymatrix{
& &\widetilde{Fl}(1,r_1,r_2,n-1;n) \ar[ld]_{\widetilde\pi_\br} \ar[rd]^{\widetilde\pi_{(n-1)1}} &\\
&Fl(r_1,r_2;n) & & (\PP^{n-1})^\vee \times \PP^{n-1}.
}
\]
We have an inclusion of tautological vector bundles $0 \to \mathcal S_1 \to \mathcal S_2$ on the flag variety $Fl(r_1,r_2;n)$.  Let $\mathcal S_2/\mathcal S_1$ be the quotient bundle.

\begin{thm}\label{thm:LVT}
With the notations as above, we have
\begin{equation}\label{eqn:KLVT}
LV\mathcal T_{\MM}(\alpha, \beta, w) =  \sum_{m=0}^{r_2-r_1} ({\widetilde \pi}_{(n-1)1})_* {\widetilde \pi}_\br^* \Big(  y(\MM) [\OO(0,1)][\bigwedge^m(\mathcal S_2/\mathcal S_1)] \Big)w^m
\end{equation}
as elements in $K^0((\PP^{n-1})^\vee\times \PP^{n-1})[w]$.
\end{thm}

We will prove the stronger statement that the $T$-equivariant version of \Cref{thm:LVT} holds.  By \Cref{prop:fundcomp}, we have that the following equality implies \Cref{thm:LVT}:
\begin{equation}\label{eqn:KLVT2}
LV\mathcal T_{\MM}(u+1,v+1,w) = \sum_{p,q,m} \chi \Big( y(\MM) [\OO(0,1)][\bigwedge^p \mathcal S_1][\bigwedge^q \mathcal Q_2^\vee][\bigwedge^m (\mathcal  S_2/\mathcal S_1)] \Big) u^pv^qw^m.
\end{equation}
We thus define the \textbf{$T$-equivariant Las Vergnas Tutte polynomial} of $\MM$ by
\[
LV\mathcal T^T_\MM(u+1, v+1, w) := \sum_{p,q,m} \chi^T \Big( y(\MM)^T  [\OO(0,1)]^T[\bigwedge^p \mathcal S_1]^T[\bigwedge^q \mathcal Q_2^\vee]^T[\bigwedge^m (\mathcal S_2/\mathcal S_1)]^T \Big) u^pv^qw^m.
\]

\begin{thm}\label{thm:equivLVT}
With the notations as above, we have
\[
LV\mathcal T^T_\MM(u+1,v+1,w) = \sum_{S\subseteq [n]} \bt^{\be_S} u^{r_1 - r_1(S)} v^{|S| - r_2(S)} w^{r_2 - r_1 - r_2(S) + r_1(S)}.
\]
\end{thm}

\begin{proof}
First, it follows from \Cref{thm:localization}.(3) that
\begin{multline} \label{eq:LVCones}
\sum_{p,q,m} \chi^T \Big( y(\MM)^T [\OO(0,1)]^T[\bigwedge^p \mathcal S_1]^T[\bigwedge^q \mathcal Q_2^\vee]^T[\bigwedge^m (\mathcal S_2/\mathcal S_1)]^T \Big) u^pv^qw^m\\
= \sum_{\tiny \begin{matrix}\BB = (B_1, B_2),\\ \BB \in \mathcal{B}(\MM)\end{matrix} } \Hilb(\Cone_\BB(\MM)) \cdot  \bt^{\be_{B_2}} \Big( \sum_{\pp' \subseteq B_1} \bt^{-\be_{\pp'}} u^{|\pp'|} \Big) \Big(\sum_{\qq\subset [n]\setminus B_2} \bt^{\be_\qq}v^{|\qq|} \Big) \Big(\sum_{\mm' \subseteq B_2\setminus B_1} \bt^{-\be_{\mm'}}w^{|\mm'|}\Big)  \\
= \sum_{\BB\in \mathcal B(\MM)} \Hilb(\Cone_\BB(\MM)) \cdot   \Big( \bt^{\be_{B_1}}\sum_{\pp' \subseteq B_1} \bt^{-\be_{\pp'}} u^{|\pp'|} \Big) \Big(\sum_{\qq\subset [n]\setminus B_2} \bt^{\be_\qq}v^{|\qq|} \Big) \Big( \bt^{\be_{B_2\setminus B_1}}\sum_{\mm' \subseteq B_2\setminus B_1} \bt^{-\be_{\mm'}}w^{|\mm'|}\Big)  \\
=\sum_{\BB \in \mathcal B(\MM)} \Hilb(\Cone_{\BB}(\MM)) \sum_{
\tiny\begin{matrix} \pp \subseteq B_1\\ \mm \subseteq B_2\setminus B_1\\ \qq \subseteq [n]\setminus B_2 \end{matrix}} \bt^{\be_\pp + \be_\mm + \be_\qq} u^{r_1 - |\pp|}v^{|\qq|}w^{r_2 - r_1 - |\mm|}.
\end{multline}

We can now compute the sum 
\[
\sum_{\BB \in \mathcal B(\MM)} \Hilb(\Cone_{\BB}(\MM)) \sum_{
\tiny\begin{matrix} \pp \subseteq B_1, |\pp|=p\\ \mm \subseteq B_2\setminus B_1, |\mm|=m\\ \qq \subseteq [n]\setminus B_2, |\qq|=q \end{matrix}} \bt^{\be_\pp + \be_\mm + \be_\qq}.
\]
for fixed $p,m,q$.

To compute the coefficient of $\bt^{\be_S}$, we apply Theorem \ref{thm:sumOfMatroidCones}.  Pick a basis $(B_1,B_2)$ such that $|S\cap B_1|=\rk_1(S)$ and $|S\cap B_2|=\rk_2(S)$. We need to compute the number of terms in the sum above for which $\BB=(B_1,B_2)$ and $\be_\pp+\be_\mm+\be_\qq=\be_S$. But such a term needs to satisfy $\pp=S \cap B_1$, $\pp \cup \mm=S\cap B_2$, and $\pp \cup \mm \cup \qq=S$. In particular $p=\rk_1(S)$, $p+m=\rk_2(S)$, and $p+m+q=|S|$. If these three equalities are satisfied, there is indeed exactly one such term. 
\end{proof}

\begin{rem}\label{rem:LVTtoT}
We remark that the Las Vergnas polynomial, and our $K$-theoretic interpretation of it, generalizes the Tutte polynomial of a matroid in the following ways.  Recall that any matroid $M$ has two canonical matroid quotients, $M\twoheadrightarrow M$ and $M \twoheadrightarrow U_{0,n}$.
\begin{itemize}
\item When $\MM = (M)$ (i.e.\ one constituent), the equation \eqref{eqn:KLVT} reduces to the one in \Cref{thm:tuttematroid} \cite[Theorem 5.1]{FS12}.
\item When $\MM = (M,M)$, one can observe from \eqref{eqn:LVT} or \eqref{eqn:KLVT} that $LV\mathcal T_\MM(x,y,z) = T_M(x,y)$.
\item When $\MM = (U_{0,n},M)$, one can observe from \eqref{eqn:LVT} or \eqref{eqn:KLVT2} that $LV\mathcal T_\MM(x,y,z) = T_M(z+1,y)$.
\end{itemize}
\end{rem}

\begin{rem}\label{rem:LVTdelcont}
The Las Vergnas Tutte polynomial satisfies a deletion-contraction relation similar to that of the Tutte polynomial \cite[Proposition 5.1]{LV80}.  We remark that our "cone-flipping with slices" (\Cref{thm:sumOfCones}) can be used to show deletion-contraction relation for $LV\mathcal T_\MM$ and $T_M$.  For example, if $i \in [n]$ is neither a loop nor a coloop of $M_2$,
\[
LV\mathcal T_{M_1, M_2}(x,y,z) = LV\mathcal T_{M_1 \setminus i, M_2 \setminus i}(x,y,z) + LV\mathcal T_{M_1/i, M_2/i}(x,y,z).
\]
This identity is obtained by applying \Cref{thm:sumOfCones} to \eqref{eq:LVCones} as follows.  By considering $\bz=\be_i$, we find that the terms in \eqref{eq:LVCones} that are not divisible by $t_i$ sum to $LV\mathcal T^{T'}_{M_1 \setminus i, M_2 \setminus i}(x,y,z)$, where $T'= (\CC^*)^{n-1}$. By considering $\bz=-\be_i$, we find that the terms that are divisible by $t_i$ sum to $t_i LV\mathcal T^{T'}_{M_1 / i, M_2 / i}(x,y,z)$. We leave the details to the reader.
\end{rem}

\begin{rem}
Unlike the Tutte polynomials of matroids, the constant term of $LV\mathcal T_\MM$ is no longer necessarily zero.  This reflects the fact that for most $\mathbf L \in Fl(r_1,r_2;n)$, the map $\widetilde\pi_{(n-1)1}: \widetilde\pi_\br^{-1}(\overline{T\cdot {\mathbf L}})\to (\PP^{n-1})^\vee \times \PP^{n-1}$ is surjective.  If further $r_2 - r_1 = 1$, then this map is a finite morphism, and by a similar computation as in \cite[Theorem 5.1]{Spe09}, one can show that the degree of the map is the Crapo's beta invariant $\beta(N)$ where $N$ is a matroid such that $M_1 = N/e$ and $M_2 = N\setminus e$.
\end{rem}

\begin{rem}
It follows from \Cref{rem:valuative} and \Cref{thm:LVT} that the assignment $\MM \mapsto LV\mathcal T_\MM$ is valuative.
\end{rem}

\section{flag-geometric Tutte polynomial of a flag matroid}\label{section:KT}

In this section, we explore the behavior of another notion of Tutte polynomials of flag matroids that differs from that of Las Vergnas in the previous section.  Here, instead of the construction \eqref{construction2}, we consider the more geometrically natural construction \eqref{construction1} in \S\ref{section:fundcomp} with the maps
\[
\xymatrix{
& &Fl(1,\br,n-1;n) \ar[ld]_{\pi_\br} \ar[rd]^{\pi_{(n-1)1}} &\\
&Fl(\br;n) & & (\PP^{n-1})^\vee \times \PP^{n-1}.
}
\]

\begin{defn}\label{defn:KTdefn} \cite[Definition 8.23]{CDMS20}
Let $\MM$ be a flag matroid of rank $\br = (r_1, \ldots, r_k)$ on $[n]$.  Then the \textbf{flag-geometric Tutte polynomial} of $\MM$, denoted $\KT_\MM(x,y) \in \ZZ[x,y]$, is the (unique) polynomial of bi-degree at most $(n-1,n-1)$ such that
\begin{equation}\label{eqn:KTdefn}
\KT_\MM(\alpha, \beta)  =  (\pi_{(n-1)1})_* \pi_r^* \Big( y (\MM) \cdot [\OO(\mathbf 1)]\Big).
\end{equation}
\end{defn}

While the construction \eqref{construction1} leading to $\KT_\MM$ may be more geometrically natural than \eqref{construction2}, the combinatorial properties of $\KT_\MM$ seem more mysterious than those of $LV\mathcal T_\MM$.  For example, in contrast to $LV\mathcal T_\MM$, the polynomial $\KT_\MM$ does not readily reduce to the Tutte polynomial of $M$ when $\MM$ is one of the two canonical matroid quotients of a matroid $M$ (i.e.\ $M\twoheadleftarrow M$ and $U_{0,n}\twoheadleftarrow M$).

\medskip
We illuminate some combinatorial structures of $\KT_\MM$ as follows.
\begin{itemize}
\item There is no known (corank-nullity) combinatorial formula for $\KT_\MM$ that is similar to \eqref{eqn:LVT} for $LV\mathcal T_\MM$.  Our result in \S\ref{subsection:uv1}, which in particular computes $\KT_\MM(2,2)$, can be considered as a first step in this direction.
\item No deletion-contraction relation is known to hold for $\KT_\MM$; one may construe this to be a consequence of the fact that the base polytope of a flag matroid generally has lattice points that are not vertices.  In \S\ref{subsection:delcont} we formulate and prove a deletion-contraction-like relation for elementary matroid quotients.
\end{itemize}

\subsection{First properties of $\KT_\MM$}\label{subsection:firstKT}

Again, by \Cref{prop:fundcomp}, we have that
\[
\KT_\MM(u+1, v+1) = \sum_{p,q} \chi \Big( y(\MM) [\OO(\mathbf 1)][\bigwedge^p \mathcal S_k][\bigwedge^q \mathcal Q_1^\vee]   \Big)u^pv^q,
\]
which leads us to the following $T$-equivariant version of $\KT_\MM$.

\begin{defn}\label{defn:FSTutte}
The $T$-equivariant flag-geometric Tutte polynomial of a flag matroid $\MM$ is
\[
\KT^T_\MM(u+1, v+1) := \sum_{p,q} \chi^T \Big( y(\MM)^T [\OO(\mathbf 1)]^T[\bigwedge^p \mathcal S_k]^T[\bigwedge^q \mathcal Q_1^\vee]^T   \Big)u^pv^q.
\]
\end{defn}

\Cref{thm:localization}.(3) again yields $\KT_\MM^T$ as a sum of rational functions as follows via a similar computation as one in the proof of \Cref{thm:LVT}.

\begin{lem}\label{lem:KTT}
For a flag matroid $\MM = (M_1, \ldots, M_k)$ on a ground set $[n]$, we have
\begin{equation}\label{eqn:KTT}
\KT^T_{\MM}(u+1,v+1)=\sum_{\BB \in \MM}{\Hilb(\Cone_{\BB}(\MM))\sum_{\pp \subseteq B_k}\sum_{\qq\subseteq [n]\setminus B_1}{\bt^{\be_{B_1}+\cdots + \be_{B_{k-1}}+ \be_\pp+\be_\qq}u^{r_k-|\pp|}v^{|\qq|}}}.
\end{equation}
\end{lem}

Many of our results on $\KT_\MM$ will be obtained by manipulation with the equation \eqref{eqn:KTT}.  We start with the following example.

\begin{eg}\label{eg:canonicalbad}
For any matroid $M$ on $[n]$, we have $\KT_{U_{0,n},M}(x,y) = y^n T_M(x,1)$\footnote{The diagram \eqref{construction1} makes sense only when $r_1\geq 1$, so $\KT_{U_{0,n},M}$ cannot be defined as a push-pull of a $K$-class.  However, we define $\KT_{U_{0,n},M}$ by specializing $\KT_{U_{0,n},M}^T$ at $t_i = 1$.}.  To verify this, we compute
\begin{equation}\label{eqn:canonicalbad}
\begin{split}
\KT_{U_{0,n},M}^T(u+1,v+1) &= \sum_{B\in \mathcal B(M)} \Hilb(\Cone_B(M)) \sum_{\pp\subseteq B}\sum_{\qq\subseteq [n]} \bt^{\be_\pp+\be_\qq}u^{r-|\pp|}v^{|\qq|}\\
&=  \Big(\prod_{i=1}^n (1+t_iv)\Big) \cdot \sum_{B\in \mathcal B(M)} \Hilb(\Cone_B(M)) \sum_{\pp\subseteq B} \bt^{\be_\pp}u^{r-|\pp|}\\
&= \Big(\prod_{i=1}^n (1+t_iv)\Big) \cdot \KT_M^T(u+1,1).
\end{split}
\end{equation}
Setting $t_i = 1$, $u = x-1$, and $v = y-1$ yields the desired claim.  This example shows that we cannot recover $T_M$ from $\KT_{U_{0,n},M}$ although $U_{0,n}\twoheadleftarrow M$ is a canonical matroid quotient of $M$.
\end{eg}

\begin{prop}\label{prop:firstproperties} Let $\MM = (M_1, \ldots, M_k)$ be a flag matroid on $[n]$.  The following properties hold for the flag-geometric Tutte polynomial $\KT^T_\MM$:
\begin{enumerate}
\item (Direct sum) If $\MM$ is a direct sum $\MM' \oplus \MM''$ of two flag matroids on ground sets $A,B$ with $A\sqcup B = [n]$, then $\KT_\MM^T(x,y) = \KT_{\MM'}^{T'}(x,y) \cdot \KT_{\MM''}^{T''}(x,y)$ (where $T' = (\CC^*)^A,$ $T'' = (\CC^*)^B$).
\item (Loops \& coloops) Let $\ell$ be the number of loops in $M_1$, and $c$ the number of coloops in $M_k$.  Then $x^c y^\ell$ divides $\KT_\MM(x,y)$.
\item (Duality) If $\MM^\vee$ is the dual flag matroid of $\MM$, whose constituents are matroid duals of the original, then $\KT_\MM(y,x) = \KT_{\MM^\vee}(x,y)$.
\item (Base polytope) $\KT_\MM^T(1,1) = \Hilb(Q(\MM))$.
\item (Valuativeness) The map $\MM\mapsto \KT_\MM$ is valuative.
\end{enumerate}
\end{prop}

\begin{proof}
The first two statements follow from manipulating with the identity \eqref{eqn:KTT} in a similar way as the computation \eqref{eqn:canonicalbad} in \Cref{eg:canonicalbad}.  For the third statement, we claim that the $T$-equivariant version of the statement is $t^{\be_{[n]}}\KT_\MM^{T^{-1}}(y,x) = \KT_{\MM^\vee}(x,y)$, where the $T^{-1}$ superscript means that we have replaced $t_i$ by $t_i^{-1}$.  Verifying this identity is then another easy manipulation with \eqref{eqn:KTT}.  The fourth statement follows from Brion's formula (\Cref{thm:generalbrion}).  The last statement follows from \Cref{rem:valuative}.
\end{proof}

We can use Theorem \ref{thm:sumOfMatroidCones} to compute some of the terms in (\ref{eqn:KTT}):
\begin{thm}
Let $\MM=(M_1,M_2)$ be a 2-step flag matroid and let $\bt^{\bk}u^{r_2-i}v^j$ be a monomial occurring in \eqref{eqn:KTT}. Then $\sum_{\ell=1}^n{k_\ell}=r_1+i+j$. Let $c$ denote the number of entries in $\bk$ that are equal to $1$. If $c \leq |r_1+j-i|$, the coefficient of $\bt^{\bk}u^{r_2-i}v^j$ is equal to 
\begin{enumerate}
	\item \label{a} $1$, if $S_2$ is spanning for $M_1$, $S_1$ is independent in $M_2$, and $c = |r_1+j-i|$,
	\item \label{b} $0$, otherwise,
\end{enumerate}
where $S_1$ and $S_2$ are defined by $S_1 \subseteq S_2$ and $\bk=\be_{S_1}+\be_{S_2}$.
\end{thm}

\begin{proof}
The equality $\sum_{\ell=1}^{n}{k_\ell}=r_1+i+j$ follows immediately from \eqref{eqn:KTT}. 
	The coefficient of $u^{r_2-i}v^j$ is equal to
	\begin{equation}\label{eqn:KTTuv}
	\sum_{\tiny \begin{matrix} \BB=(B_1, B_2),\\ \BB \in \mathcal B(\MM)\end{matrix}}{\Hilb(\Cone_\BB(\MM))\sum_{\tiny \begin{matrix} \pp \subseteq B_2, \\ |\pp|=i\end{matrix}}\sum_{\tiny \begin{matrix} \qq\subseteq J_1,\\  |\qq|=j\end{matrix}}{\bt^{\be_{B_1}+\be_\pp+\be_\qq}}},
	\end{equation}
where we have denoted $J_1 := [n]\setminus B_1$.  The vertices of $\widetilde{P}$ have $|r_1+j-i|$ entries equal to $1$. This proves that the coefficient is $0$ if $c < |r_1+j-i|$. So from now on we assume $c = |r_1+j-i|$.
	
	Next, we apply \Cref{thm:sumOfMatroidCones}. 
	Writing $\bk=\be_{S_1}+\be_{S_2}$, and noting that $|S_1|=\min(i,r_1+j)$ and $|S_2|=\max(i,r_1+j)$, we find a basis $\BB = (B_1,B_2)$ of $\MM$ for which $r_j(S_i)=|S_i \cap B_j|$. We now need to compute the number of ways $\bk$ can be written as a sum $\be_{B_1}+\be_\pp+\be_\qq$.
	If $S_2$ is not spanning for $M_1$, or if $S_1$ is not independent in $M_2$, there are no ways to do this, and the coefficient is $0$.
	Otherwise, if $i \leq r_1+j$, we need to put $\pp=S_1$ and $\qq=S_2 \setminus S_1$. If $i \geq r_1+j$, we need to put $\qq=S_1 \cap J_1$ and $\pp=S_1 \cup J_1$. In both cases, there is just one way, so the coefficient is $1$.
\end{proof}

\subsection{Towards a corank-nullity formula}\label{subsection:uv1}

For a matroid $M$ on $[n]$, the corank-nullity formula for the Tutte polynomial
$T_M(x,y) = \sum_{S\subseteq [n]} (x-1)^{r - r(S)}(y-1)^{|S| - r(S)}$
expresses $T_M$ as a sum over all subsets of $[n]$.  In particular, we have $T_M(2,2) = 2^n$; in fact, $\KT^T_M(2,2) = \prod_{i=1}^n (1+t_i)$.  As a first step towards a similar formula for $\KT_\MM$, we show the following for a two-step flag matroid.

\begin{thm}\label{thm:KT22}
Let $\MM$ be a two-step flag matroid $\MM = (M_1,M_2)$ of rank $(r_1, r_2)$, and let $p\mathcal B(\MM)$ be the set of pseudo-bases of $\MM$, i.e.\ subsets $S\subseteq [n]$ such that $S$ is spanning in $M_1$ and independent in $M_2$.  With $q$ as a formal variable, we have
\[
\KT_{\MM}^T(1+q^{-1},1+q) = q^{-r_2} \Big( \prod_{i =1}^{n}{(1+t_i q)} \Big) \Big(\sum_{S \in p\mathcal B(\MM)} \bt^{\be_S} q^{|S|}\Big),
\]
and in particular, we have
\[
\begin{split}
& \KT_\MM(1+q^{-1},1+q) = q^{-r_2} \cdot (1+q)^n \cdot \Big(\sum_{S \in p\mathcal B(\MM)} q^{|S|}\Big),\\
& \KT_\MM^T(2,2) = \Big( \prod_{i=1}^n (1+t_i) \Big) \Big(\sum_{S \in p\mathcal B(\MM)} \bt^{\be_S} \Big), \textnormal{ and}\\
& \KT_\MM(2,2) = 2^n|p\mathcal B(\MM)|.
\end{split}
\]
\end{thm}

\begin{proof}
Setting $u = q^{-1}$ and $v = q$ in \eqref{eqn:KTT} of \Cref{lem:KTT} gives us
\begin{align*}
\KT^T_{\MM}(1+q^{-1},1+q) &= \sum_{\tiny \begin{matrix} \BB=(B_1, B_2),\\\BB \in \mathcal B(\MM) \end{matrix} }{\Hilb(\Cone_\BB(\MM))\sum_{\pp \subseteq B_2}\sum_{\qq \subseteq [n]\setminus B_1}{\bt^{\be_{B_1}+\be_\pp+\be_\qq}q^{|\pp|+|\qq|-r_2}}}\\
&= \sum_{\tiny \begin{matrix} \BB=(B_1, B_2),\\\BB \in \mathcal B(\MM) \end{matrix} }{\Hilb(\Cone_\BB(\MM))\sum_{R \subseteq E}\sum_{S\subseteq B_2 \setminus B_1}{\bt^{\be_{B_1}+\be_R+\be_S}q^{|R|+|S|-r_2}}}\\
&= q^{-r_2}\prod_{i \in E}{(1+t_iq)}\sum_{\tiny \begin{matrix} \BB=(B_1, B_2),\\\BB \in \mathcal B(\MM) \end{matrix} }{\Hilb(\Cone_B(\MM))\sum_{S\subseteq B_2 \setminus B_1}{\bt^{\be_{B_1}+\be_S}q^{|S|}}}.
\end{align*}
We now use Theorem \ref{thm:sumOfMatroidCones} to compute the sum 
\[
\varphi_r := \sum_{\tiny \begin{matrix} \BB=(B_1, B_2),\\\BB \in \mathcal B(\MM) \end{matrix} }{\Hilb(\Cone_\BB(\MM))\sum_{\tiny \begin{matrix} B_1 \subseteq \pp \subseteq B_2,\\ |\pp|= r \end{matrix}}\bt^{\be_\pp}}
\]
for a fixed $r_1 \leq r \leq r_2$.  First, we note that the polytope $\widetilde{P} = \operatorname{Conv}(\be_S \mid S\subseteq E,\ |S| = r)$, obtained as the convex hull of the $S_n$-orbit of $\{\be_\pp \mid B_1\subseteq \pp \subseteq B_2,\ |\pp| = r\}$, has no interior lattice points.

For $S \subseteq E$ with $|S| = r$, if $S$ is not a pseudo-basis of $M_1 \twoheadleftarrow M_2$, then there is no basis $\BB$ of $\MM$ such that $B_1 \subseteq S \subseteq B_2$, and hence the coefficient of $\bt^{\be_S}$ is $0$ in this case. Now, suppose $S$ is a pseudo-basis of $M_1 \twoheadleftarrow M_2$, which by definition implies that there exists basis $\BB=(B_1,B_2)$ of $\MM$ with $B_1 \subseteq S \subseteq B_2$. This basis $\BB$ is a vertex of the face $Q(\MM)^{\be_S}$ by \Cref{prop:face}, and thus by \Cref{thm:sumOfMatroidCones} the coefficient of $\bt^{\be_S}$ is equal to 1 in $\varphi_r$.
\end{proof}

We do not know of analogues of \Cref{thm:KT22} for flag matroids with more than two constituents.

\subsection{A deletion-contraction-like relation}\label{subsection:delcont}

In this section, we consider $\KT_\MM$ of an elementary quotient $\MM = (M_1, M_2)$.  By definition we have $r(M_2) - r(M_1) = 1$, and in this case there is a unique matroid $M$ on a ground set $[\widetilde n] := \{0\}\sqcup [n]$ such that $M_1 = M/0$ and $M_2 = M\setminus 0$ \cite[\S7.3]{Oxl11}.  Our main theorem of this subsection is the following deletion-contraction-like relation.

\begin{thm}\label{thm:delcont}
Let $M$ be a matroid of rank $r$ on $[\widetilde n] := \{0\} \sqcup [n]$ such that the element 0 is neither a loop nor a coloop in $M$.  Let $\widetilde T = \CC^* \times T = (\CC^*)^{n+1}$ be the torus with character ring $\ZZ[t_0^\pm, \ldots, t_n^\pm]$.  Then we have
\begin{equation}\label{eqn:delcontrelation}
\KT^{\widetilde T}_{M,M}(x,y) = t_0^2 \KT^T_{M/0,M/0}(x,y) + t_0 \KT^T_{M/0,M\setminus 0}(x,y)  + \KT^T_{M\setminus 0, M\setminus 0}(x,y).
\end{equation}
In particular, we have $\KT_{M,M}(x,y) =  \KT_{M/0,M/0}(x,y) + \KT_{M/0,M\setminus 0}(x,y)  + \KT_{M\setminus 0, M\setminus 0}(x,y).$
\end{thm}

We use $\{\be_0, \ldots, \be_n\}$ for the standard basis of $\RR^{n+1} = \RR\oplus \RR^n$.  For a polyhedron $P \subset \RR^n$, we will often abuse the notation and write $P$ also for $\{\mathbf 0 \}\times P \subset \RR \oplus \RR^n$.  We prepare for the proof of \Cref{thm:delcont} by an observation that motivated the theorem.  

\medskip
As the base polytope $Q(M)$ is a $(0,1)$-polytope (i.e.\ a lattice polytope contained in the Boolean cube $[0,1]^{n+1} \subset \RR^{n+1}$), every lattice point is a vertex.  Moreover, observe that the vertices of $Q(M)$ partition into two parts, the bases of $M/0$ and the bases of $M\setminus 0$.
As a result, the lattice points of $Q(M,M) = Q(M) + Q(M)$ partition into the following three parts, with $Q_1 = \frac{1}{2} (Q_0 + Q_2)$:
\begin{itemize}
\item $Q_2 := Q(M,M) \cap H_{\be_0 = 2} = \{2\be_0\} \times Q(M/0,M/0)$,
\item $Q_1 := Q(M,M) \cap H_{\be_0 = 1} = \{\be_0\} \times Q(M/0,M\setminus 0)$, and
\item $Q_0 := Q(M,M) \cap H_{\be_0 = 0} = \{\mathbf 0\}\times Q(M\setminus 0, M \setminus 0)$.
\end{itemize}
The case of setting $x = y = 1$ (cf.\ \Cref{prop:firstproperties}.(4)) in \eqref{eqn:delcontrelation} of \Cref{thm:delcont} witnesses this partition of the lattice points of $Q(M,M)$ .  The following lemma in preparation for the proof of \Cref{thm:delcont} is a consequence of $Q_1 = \frac12 (Q_0 + Q_2)$.

\begin{lem}\label{lem:delcont}
Let the notations be as above.  Then for $B\in \mathcal B(M)$ with $0\notin B$, we have
\[
\Hilb(\Cone_B(Q(M,M)) \cap H_{\be_0 = 1}) = \sum_{\tiny\begin{matrix} I \in \mathcal B(M/0), \\ I \subset B\end{matrix}} t_0t_{B\setminus I}^{-1} \Hilb(\Cone_{(I,B)}(Q(M/0,M\setminus 0))
\]
and
\[
\Hilb(\Cone_B(Q(M,M)) \cap H_{\be_0 = 0}) = \sum_{\tiny\begin{matrix} I \in \mathcal B(M/0), \\ I \subset B\end{matrix}} \Hilb(\Cone_{(I,B)}(Q(M/0,M\setminus 0)).
\]
\end{lem}

\begin{proof}
We have an equality of polyhedra
\[
\Cone_B(Q(M,M)) \cap H_{\be_0 = 1} = \Cone_B(Q(M\setminus 0)) + Q_1 - 2\be_B.
\]
We claim that $\Cone_B(Q(M\setminus 0)) + Q_1$ has vertices $\{\be_I + \be_B\}$ for $I \in \mathcal B(M/0)$ such that $I\subset B$.  The two statements in the lemma then follow from Brion's formula \Cref{thm:generalbrion}.

For the claim, we start by noting that if $I \in \mathcal B(M/0)$ then there exists $B' \in \mathcal B(M\setminus 0)$ such that $I\subset B'$ (since $M/0 \twoheadleftarrow M\setminus 0$).  Consequently, if $\be_B$ is the vertex of $Q(M\setminus 0)$ that minimizes $\langle v, \be_B\rangle$ for some $v\in \RR^n$, then a vertex of $Q(M/0)$ that minimizes $\langle v, \cdot\rangle$ must be $\be_I$ satisfying $I\subset B$.  Our claim now follows from $Q_1 = \frac12 (Q_0 + Q_2)$.
\end{proof}

\begin{proof}[Proof of \Cref{thm:delcont}]
  Let us begin by noting that the equation \eqref{eqn:KTT} for $\KT^{\widetilde T}_{M,M}$ reads
\begin{equation}\label{eqn:KTMM}
\KT^{\widetilde T}_{M,M}(u+1, v+1) = \sum_{B\in \mathcal B(M)} \Hilb(\Cone_B(Q(M,M))) \sum_{\pp \subseteq B}\sum_{\qq \subseteq [\widetilde n]\setminus B} \bt^{\be_B + \be_\pp + \be_\qq}u^{r-|\pp|}v^{|\qq|}.
\end{equation}
We apply \Cref{thm:sumOfCones} with $\bz = \be_0$ and $L$ defined by $t_0=0$.
Note that $\Cone_B(Q(M,M)) \in \fP_n^{\bz}$ if and only if $0 \notin B$. Hence all cones occurring in \eqref{eqn:KTMM} with vertex on $L$ are in $\fP_n^{\bz}$, and we find that the terms in \eqref{eqn:KTMM} not divisible by $t_0$ sum to
\begin{multline*}
\sum_{\tiny\begin{matrix} B \in \mathcal B(M),\\ 0 \notin B\end{matrix}}  \Hilb(\Cone_B(Q(M,M))\cap H_{\be_0 = 0}) \sum_{\pp \subseteq B}\sum_{\qq \subseteq [n]\setminus B} \bt^{\be_B + \be_\pp + \be_\qq}u^{r-|\pp|}v^{|\qq|}\\
= \sum_{B\in \mathcal B(M\setminus 0)} \Hilb(\Cone_{B}(M\setminus 0))\sum_{\pp\subseteq B}\sum_{\qq\subseteq [n]\setminus B} \bt^{\be_B + \be_\pp + \be_\qq}u^{r-|\pp|}v^{|\qq|} \\
= \KT^T_{M\setminus 0, M\setminus 0}(u+1, v+1).
\end{multline*}

A similar argument, with $\bz = - \be_0$, shows that the coefficient of $t_0^2$ in \eqref{eqn:KTMM} is $\KT^T_{M/0,M/0}$.

Finally, we apply \Cref{thm:sumOfCones} once more, this time with $\bz=\be_0$ and $L=H_{\be_0=1}$. We find that the terms in \eqref{eqn:KTMM} divisible by $t_0$ but not by $t_0^2$ sum to
\[
\begin{split}
&\Big( \sum_{\tiny\begin{matrix} B \in \mathcal B(M),\\ 0 \notin B\end{matrix}}  \Hilb(\Cone_B(Q(M,M))) \sum_{\pp \subseteq B}\sum_{\qq \subseteq [\widetilde n]\setminus B} \bt^{\be_B + \be_\pp + \be_\qq}u^{r-|\pp|}v^{|\qq|} \left.\Big)\right|_{H_{\be_0 = 1}}\\
&= \Big( \sum_{\tiny\begin{matrix} B \in \mathcal B(M),\\ 0 \notin B\end{matrix}}  \Hilb(\Cone_B(Q(M,M))) \sum_{\pp \subseteq B}\sum_{\qq \subseteq [n]\setminus B} \bt^{\be_B + \be_\pp + \be_\qq}(1+t_0v)u^{r-|\pp|}v^{|\qq|} \left.\Big)\right|_{H_{\be_0 = 1}}\\
&= \sum_{\tiny\begin{matrix} B \in \mathcal B(M),\\ 0 \notin B\end{matrix}}  \Hilb(\Cone_B(Q(M,M))\cap H_{\be_0 = 1}) \sum_{\pp \subseteq B}\sum_{\qq \subseteq [n]\setminus B} \bt^{\be_B + \be_\pp + \be_\qq}u^{r-|\pp|}v^{|\qq|} \\
&\qquad + t_0 \sum_{\tiny\begin{matrix} B \in \mathcal B(M),\\ 0 \notin B\end{matrix}}  \Hilb(\Cone_B(Q(M,M))\cap H_{\be_0 = 0}) \sum_{\pp \subseteq B}\sum_{\qq \subseteq [n]\setminus B} \bt^{\be_B + \be_\pp + \be_\qq}u^{r-|\pp|}v^{|\qq|+1},\\
\end{split}
\]
which by \Cref{lem:delcont} is equal to
\[
\begin{split}
& \sum_{\tiny\begin{matrix} B \in \mathcal B(M),\\ 0 \notin B\end{matrix}}  \sum_{\tiny\begin{matrix} I \in \mathcal B(M/0),\\ I\subset B\end{matrix}}  t_0t_{B\setminus I}^{-1} \Hilb(\Cone_{(I,B)}(Q(M/0,M\setminus 0)))  \sum_{\pp \subseteq B}\sum_{\qq \subseteq [n]\setminus B} \bt^{\be_B + \be_\pp + \be_\qq}u^{r-|\pp|}v^{|\qq|}\\
&\qquad + t_0\sum_{\tiny\begin{matrix} B \in \mathcal B(M),\\ 0 \notin B\end{matrix}}  \sum_{\tiny\begin{matrix} I \in \mathcal B(M/0),\\ I\subset B\end{matrix}}  \Hilb(\Cone_{(I,B)}(Q(M/0,M\setminus 0)))  \sum_{\pp \subseteq B}\sum_{\qq \subseteq [n]\setminus B} \bt^{\be_B + \be_\pp + \be_\qq}u^{r-|\pp|}v^{|\qq|+1}\\
&= t_0 \sum_{(I,B) \in \mathcal B(M/0,M\setminus 0)} \Hilb(\Cone_{(I,B)}(Q(M/0,M\setminus 0))) \Big( \sum_{\pp \subseteq B}\sum_{\qq \subseteq [n]\setminus B} \bt^{\be_I + \be_\pp + \be_\qq}(1+t_{B\setminus I}v) \Big) u^{r-|\pp|}v^{|\qq|}\\
& = t_0 \sum_{(I,B) \in \mathcal B(M/0,M\setminus 0)} \Hilb(\Cone_{(I,B)}(Q(M/0,M\setminus 0))) \Big( \sum_{\pp \subseteq B}\sum_{\qq \subseteq [n]\setminus I} \bt^{\be_I + \be_\pp + \be_\qq} \Big) u^{r-|\pp|}v^{|\qq|}\\
&= t_0\KT^T_{M/0,M\setminus 0}(u+1, v+1),
\end{split}
\]
as desired.
\end{proof}

\begin{rem}
We remark that for a general flag matroid $\MM$, the slices $\{Q(\MM) \cap H_{\be_i = k}\}_{k\in \ZZ}$ need not be flag matroid base polytopes.  Moreover, even when they are, we do not observe an identity like the one in \Cref{thm:delcont} that expresses $\KT_\MM$ in terms of the slices.

For example, consider $\MM = (U_{1,3}, U_{2,3})$.  We have $\KT_\MM(x,y) = x^2y^2 + x^2y+xy^2 + x^2 + 2xy + y^2$.  In any coordinate direction, its three slices are $(U_{0,2},U_{1,2})$, $(U_{1,2},U_{1,2})$, and $(U_{1,2},U_{2,2})$, whose $\KT$ are (respectively), $xy^2+y^2$, $xy+x+y$, and $x^2y+x^2$.
\end{rem}

\begin{rem}
One can generalize \Cref{thm:delcont} as follows.  Denote by $M^\ell := (M, \ldots, M)$, the flag matroid whose constituents are $M$ repeated $\ell$ times.  Then we have
\[
\KT^{\widetilde T}_{M^\ell} = t_0^\ell \KT^T_{(M/0)^\ell} + t_0^{\ell-1}\KT^T_{(M/0)^{\ell-1},M\setminus 0}  + \cdots + \KT_{(M\setminus 0)^\ell}.
\]
The proof is essentially identical to one given for \Cref{thm:delcont}. 
\end{rem}

\section{Future directions}\label{section:future}

We list two future directions stemming from our work here.

\subsection{$g$ and $h$ polynomials for flag matroids}

For a matroid $M$, Speyer introduced in \cite{Spe09} a polynomial invariant $g_M(t) \in \QQ[t]$ and a close cousin $h_M(t) \in \QQ[t]$, which is related to $g_M(t)$ by $ h_M(t) = (-1)^c g_M(-t)$ where $c$ is the number of connected components of $M$.  A $K$-theoretic interpretation of the polynomial $h_M$ was given in \cite{FS12}.

\begin{thm}\cite[Theorem 6.1 \& Theorem 6.5]{FS12}
Let $M$ be a matroid of rank $r$ on $[n]$ without loops or coloops.  Let $\pi_r, \pi_{(n-1)1}, \alpha, \beta$ be as in \S\ref{subsection:flagmatKclass}.  Then the polynomial $h_M$ is the (unique) univariate polynomial of degree at most $n-1$ such that
\[
(\pi_{(n-1)1})_* \pi_r^* \Big(y(M)\Big) = h_M(\alpha + \beta - \alpha\beta).
\]
\end{thm}

For a flag matroid $\MM$ on $[n]$, this motivates us to consider $(\pi_{(n-1)1})_* \pi_\br^* \Big(y(\MM)\Big)$, where the maps are as in the flag-geometric construction \eqref{construction1}.  By \Cref{prop:fundcomp}, this is equal to
\[\sum_{p,q}\chi \Big( y(\MM) [\bigwedge^p \mathcal S_k][\bigwedge^q \mathcal Q_1^\vee]  \Big)(\alpha-1)^p(\beta-1)^q.
\]
Let us consider its torus-equivariant version
\[
\sum_{p,q} \chi^T \Big( y(\MM)^T [\bigwedge^p \mathcal S_k]^T[\bigwedge^q \mathcal Q_1^\vee]^T   \Big)u^pv^q
\]
where $u$ and $v$ are formal variables.  We show that this is a polynomial in $uv$, which thereby establishes that $(\pi_{(n-1)1})_* \pi_\br^* \Big(y(\MM)\Big)$ is a polynomial in $\alpha+\beta - \alpha\beta$ (since the substitution $u= \alpha-1, v = \beta -1$ yields $1 -uv = \alpha + \beta - \alpha\beta$).

\begin{lem}\label{lem:inuv} (cf.\ \cite[Lemma 6.2]{FS12})
Let $\MM=(M_1, \ldots, M_k)$ be a flag matroid on $[n]$, and suppose every constituent of $\MM$ is both loopless and coloopless.  Then
\[
\sum_{p,q} \chi^T \Big( y(\MM)^T [\bigwedge^p \mathcal S_k]^T[\bigwedge^q \mathcal Q_1^\vee]^T   \Big)u^pv^q \in \QQ[u,v]
\]
is a polynomial in $\QQ[uv]$.
\end{lem}

We remark that the condition about a flag matroid $\MM = (M_1, \ldots, M_k)$ being loopless or coloopless depends only on $M_1$ or $M_k$, respectively.  First, note that by the condition (2) in \Cref{defn:matroidquotient}, if $\ell\in [n]$ is a loop in $M_i$ then it is a loop in $M_{i-1}$ also.  By duality, if $\ell\in [n]$ is a coloop in $M_i$ then it is a coloop in $M_{i+1}$ also.  Hence, the flag matroid $\MM$ is loopless (coloopless) if and only if $M_1$ has no loops ($M_k$ has no coloops).

\begin{proof}
Once more by \Cref{thm:localization}.(3), we get
\[
\sum_{p,q} \chi^T \Big( y(\MM)^T [\bigwedge^p \mathcal S_k]^T[\bigwedge^q \mathcal Q_1^\vee]^T   \Big)u^pv^q = \sum_{\BB \in \MM}{\Hilb(\Cone_{\BB}(\MM))\sum_{\pp \subseteq B_k}\sum_{\qq\subseteq [n]\setminus B_1}{\bt^{-\be_\pp+\be_\qq}u^{|\pp|}v^{|\qq|}}}.
\]
Fix $|\pp|=i, |\qq|=j$, and consider the sum 
\begin{equation}
\label{eqn:HTterm}
\varphi_{ij}=\sum_{\BB \in \MM}{\Hilb(\Cone_{\BB}(\MM))\sum_{\tiny\begin{matrix} \pp \subseteq B_k,\\ |\pp|=i\end{matrix}}\sum_{\tiny\begin{matrix} \qq \subseteq [n] \setminus B_1,\\ |\qq|=j\end{matrix}}{\bt^{-\be_\pp+\be_\qq}}}.
\end{equation}
We need show that $\varphi_{ij}$ is zero if $i\neq j$.  Let $P$ be the convex hull of $\{-\be_\pp+\be_\qq\}$ appearing in the summation \eqref{eqn:HTterm}.  Note that $P$ is contained in the intersection of $H_{\be_{[n]} = j-i}$ and the cube $\{ x\in \RR^n \mid -1 \leq x_\ell \leq 1 \ \forall \ell \in [n]\}$.  By \Cref{cor:vertexcase}.\ref{vertexcase:2}, it thus suffices to show that $\varphi_{ij}|_{H_{\be_\ell = -1}} = 0$ and $\varphi_{ij}|_{H_{\be_\ell = 1}} = 0$ for all $\ell \in [n]$.

Let us now fix any $\ell\in [n]$.
As none of the constituents have coloops (and in particular $\ell$ is not a coloop in $M_k$), the intersection $Q(\MM) \cap H_{\be_\ell = 0}$ is a non-empty face of $Q(\MM)$ minimizing in the $\be_\ell$ direction, consisting of bases $\BB = (B_1, \ldots B_k)$ such that $\ell \notin B_k$.  Thus, we have that $\Cone_{\BB}(\MM) \in \fP_n^{\be_\ell}$ if and only if $\ell \notin B_k$, and by \Cref{thm:sumOfCones} with $\bz = \be_\ell$ we have
\[
\varphi_{ij}|_{H_{\be_\ell=-1}} = \sum_{\BB: \ell \notin B_k}\sum_{\tiny\begin{matrix} \pp \subseteq B_k,\\ |\pp|=i\end{matrix}}\sum_{\tiny\begin{matrix} \qq \subseteq [n] \setminus B_1,\\ |\qq|=j\end{matrix}}{\Hilb((-\be_\pp+\be_\qq+\Cone_{\BB}(\MM))|_{H_{\be_\ell =-1}})}.
\]
But since $\ell \notin B_k$ implies $\ell \notin \pp$, every cone $-\be_\pp+\be_\qq+\Cone_{\BB}(\MM)$ occurring in the sum above will have vertex $v$ with $v_\ell >-1$.  Hence, noting that $\Cone_{\BB}(\MM) \in \fP_n^{\be_\ell}$ for such cones, we have $\varphi_{ij}|_{H_{\be_\ell=-1}}=0$. A similar argument with $\bz=-\be_\ell$, noting that $\ell$ is not a loop in $M_1$, shows that $\varphi_{ij}|_{H_{\be_\ell=1}}=0$.
\end{proof}

We thus make the following definition that generalizes the polynomial $h_M$ of a matroid $M$ to the setting of flag matroids.  It is well-defined by \Cref{lem:inuv}.

\begin{defn}
Let $\MM=(M_1, \ldots, M_k)$ be a flag matroid $[n]$ such that every constituent of $\MM$ is both loopless and coloopless.  Let $\pi_{(n-1)1}, \pi_\br, \alpha, \beta$ be as in \S\ref{section:fundcomp}.  Then the polynomial $h_\MM$ is defined as the (unique) univariate polynomial of degree at most $n-1$ such that
\[
(\pi_{(n-1)1})_* \pi_\br^* \Big(y(\MM)\Big) = h_\MM(\alpha+\beta-\alpha\beta).
\]
\end{defn}

\begin{rem}\label{rem:LVHbad}
We have constructed the polynomial $h_\MM$ via the flag-geometric diagram \eqref{construction1}.  Although one may also consider a similar construction via the "Las Vergnas" diagram \eqref{construction2}, a computer computation (\S\ref{subsection:computer}) shows that the analogue of \Cref{lem:inuv} fails in this case, for instance with $\MM = (U_{2,4},U_{3,4})$.
\end{rem}

In the case of matroids realizable over $\CC$, the behavior of the polynomial $g_M$ of a matroid $M$, in particular the non-negativity of its coefficients, was used to establish a bound on the number of interior faces in a matroidal subdivision of a base polytope of a matroid \cite{Spe09}.  Extending these results to arbitrary matroids is so far open, but an announcement of a relevant forthcoming work has been made in \cite{LdMRS20}.

\medskip
In another work \cite{BEZ20}, the authors study flag-matroidal subdivisions of base polytopes of flag matroids, and extend the tropical geometry of matroids used in \cite{Spe09} to the setting of flag matroids.  We are thus led to ask the following.

\begin{ques}
Does a suitable modification of our polynomial $h_\MM$ give an analogue of the polynomial $g_M$ for flag matroids, and does its behavior lead to a bound on the number of interior faces in a flag-matroidal subdivision of a base polytope of a flag matroid?
\end{ques}

\subsection{Characteristic polynomials of matroid morphisms}

A recent breakthrough in matroid theory is the log-concavity of the coefficients of the characteristic polynomial of a matroid \cite{AHK18}.  We consider here several candidates for characteristic polynomials of morphisms of matroids.  We begin with the one coming from the flag-geometric Tutte polynomial.

\begin{defn} For a flag matroid $\MM$, define the \textbf{flag-geometric characteristic polynomial} $K\chi_\MM(q)$ of $\MM$ by
\[
K\chi_\MM(q) := (-1)^{r_k}\KT_\MM(1-q,0).
\]
\end{defn}

Like the usual characteristic polynomial, the polynomial $\KT_\MM$ satisfies $K\chi_\MM(q) = 0$ whenever the first constituent $M_1$ of $\MM$ has a loop by \Cref{prop:firstproperties}.  The following conjecture is supported by computer computations (\S\ref{subsection:computer}).  It suggests that the flag-geometric characteristic polynomial of a two-step flag matroid may contain little information about the flag matroid itself.

\begin{conj}
Let $M$ be a matroid of rank $r$ with no loops, so that $U_{1,n} \twoheadleftarrow M$ is a valid matroid quotient.  Then $K\chi_{(U_{1,n},M)}(q) = (q-1)^r$.
\end{conj}

Let us now turn to the Las Vergnas Tutte polynomial.  The last two bullet points of \Cref{rem:LVTtoT} suggest two different ways of generalizing the characteristic polynomial of a matroid.  The case of $\MM = (M,M)$ gives rise to the polynomial $p_\MM(q,s) := (-1)^{r_2} LV\mathcal T_\MM(1-q,0,-s)$, which was studied by Las Vergnas as the \textbf{Poincar\'e polynomial} of a matroid quotient \cite[\S4]{LV80}.  Here we introduce another generalization following the case of $\MM = (U_{0,n},M)$.

\begin{defn}
For a flag matroid $\MM = (M_1, M_2)$ on $[n]$, define its \textbf{beta polynomial} $\beta_\MM(q)$ by
\[
\beta_\MM(q) := (-1)^{r_2 - r_1}LV\mathcal T_\MM(0,0,-q).
\]
\end{defn}

When $\MM = (U_{0,n},M)$, it follows from $LV\mathcal T_\MM(x,y,z) = T_M(z+1,y)$ that $\beta_\MM(q) = \chi_M(q)$, the characteristic polynomial of $M$.  The terminology for $\beta_\MM(q)$ is motivated by \Cref{prop:betapolynomial} below.  First, let us recall that the \textbf{beta invariant} $\beta(M)$ of a matroid $M$ of rank $r$ is defined as
\[
\beta(M) := \textstyle(-1)^{r-1} \left( \left.\frac{d}{dq}\chi_M(q)\right|_{q=1}\right),
\]
and that if $e$ is an element that is neither a loop nor a coloop, then $\beta(M/e) + \beta(M\setminus e) = \beta(M)$ \cite{Cra67}.

\begin{prop}\label{prop:betapolynomial}
Let $M_1 = M^{(r_2 - r_1)}\twoheadleftarrow \cdots\twoheadleftarrow M^{(1)} \twoheadleftarrow M^{(0)} = M_2$ be the Higgs factorization of a matroid quotient $M_1 \twoheadleftarrow M_2$ as described in \S\ref{subsection:flagmat}.  The beta polynomial $\beta_{M_1,M_2}(q)$ is divisible by $(q-1)$, and the \textbf{reduced beta polynomial} of $M_1 \twoheadleftarrow M_2$, defined as $\overline{\beta}_{M_1,M_2}(q) := \beta_{M_1,M_2}(q)/(q-1)$, satisfies
\[
\overline \beta_{M_1,M_2}(q) = \sum_{i = 0}^{r_2 - r_1 -1} (-1)^{r_2 - r_1 -1 - i} \big(\beta(M^{(i)}) + \beta(M^{(i+1)}) \big) q^i.
\]
\end{prop}

If $\widetilde{M}^{(i)}$ is the (unique) matroid on $[n]\sqcup \{0\}$ such that $\widetilde{M}^{(i)} / 0 = {M}^{(i+1)}$ and $\widetilde{M}^{(i)} \setminus 0 = {M}^{(i)}$, which exists by \cite[\S7.3]{Oxl11}, then $\beta({M}^{(i)}) + \beta({M}^{(i+1)}) = \beta(\widetilde{M}^{(i)})$.  So, the proposition says that 
\[
\overline{\beta}_{M_1,M_2}(q) = \sum_{i=0}^{d-1} (-1)^{d-1-i}\beta(\widetilde{M}^{(i)}) \qquad \text{where $d = r_2 - r_1$}.
\]

\begin{proof}
Let $\MM = (M_1, M_2)$ and $d = r_2 - r_1$.  \cite[Theorem 3.1]{LV80} states that
\[
LV\mathcal T_\MM(x,y,z) = \sum_{i=0}^d t_i(\MM;x,y)z^i,
\]
where
\[
t_i(\MM;x,y) = \frac{1}{xy-x-y} \Big((y-1)T_{M^{(i-1)}}(x,y) + (-xy+x+y-2)T_{M^{(i)}}(x,y)+ (x-1)T_{M^{(i+1)}}(x,y) \Big)
\]
for $i = 1, \ldots, d-1$, and
\[
\begin{split}
t_0(\MM; x,y)= \frac{1}{xy-x-y} \Big( -T_{M^{(0)}}(x,y) + (x-1)T_{M^{(1)}}(x,y) \Big),\\
t_d(\MM;x,y) = \frac{1}{xy-x-y} \Big( (y-1)T_{M^{(d-1)}}(x,y) - T_{M^{(d)}}(x,y) \Big).
\end{split}
\]
Let us express the beta invariant $\beta(M)$ of a matroid $M$ of rank $r$ equivalently as
\[
\begin{split}
\beta(M) &= \textstyle(-1)^{r-1} \left( \left.\frac{d}{dq}\chi_M(q)\right|_{q=1}\right),\\
&=  (-1)^{r-1}\lim_{q\to 1} \frac{(-1)^rT_M(1-q,0) - (-1)^rT_M(0,0)}{q-1} = - \lim_{q\to 1} \frac{T_M(1-q,0)}{q-1}.
\end{split}
\]
As $LV\mathcal T_\MM(x,y,z)$ is a polynomial, each $t_i(\MM;x,y)$ is also a polynomial.  Hence, we have $t_i(\MM;0,0) = \lim_{q\to 1} t_i(\MM; 1-q,0)$, and thus the above expressions for $t_i(\MM;x,y)$ give
\[
\begin{split}
&t_i(\MM;0,0) = \beta(M^{(k-1)}) + 2\beta(M^{(k)}) + \beta(M^{(k+1)}) \quad \textnormal{for }i = 1, \ldots, d-1, \textnormal{ and}\\
&t_0(\MM;0,0) = \beta(M^{(0)}) + \beta(M^{(1)}),\\
&t_d(\MM;0,0) = \beta(M^{(d-1)}) + \beta(M^{(d)}).
\end{split}
\]
As a result, we have
\[
(-1)^d LV\mathcal T_\MM(0,0,-q) = (q-1) \Big(\sum_{i = 0}^{d-1} (-1)^{d-1-i}\big(\beta(M^{(i)}) + \beta(M^{(i+1)}) \big)q^i\Big),
\]
yielding the desired result for the reduced beta polynomial $\overline{\beta}_\MM(q)$.
\end{proof}

Log-concavity of the coefficients of the reduced characteristic polynomial $\overline{\chi}_M(q) = \frac{\chi_M(q)}{q-1}$ was established in \cite{AHK18}.  This motivates the following conjecture.  

\begin{conj}
The coefficients of $\overline \beta_{M_1,M_2}(q)$ form a log-concave sequence.  Consequently, the coefficients of $LV\mathcal T_\MM(0,0,q)$ form a log-concave sequence.
\end{conj}

The coefficients of $LV\mathcal T_\MM(1,1,q)$ were shown to be (ultra) log-concave in \cite{EH20}.

\subsection*{Acknowledgements}

The authors thank Alex Fink and Mateusz Micha{\l}ek for helpful conversations.  The authors also thank the hospitality of the Max Planck Institute for Mathematics in the Sciences (MiS) in Leipzig.  Lastly, the authors thank the anonymous referee for a careful reading and helpful suggestions.  C.E.\ is partially supported by US National Science Foundation (DMS-2001854).

\bibliography{DinuEurSeynnaeve_KTutteMatMorph_v2}
\bibliographystyle{alpha}

\end{document}